\documentclass[12pt]{amsart}
\usepackage[utf8]{inputenc} 
\usepackage[active]{srcltx}
\usepackage{a4wide}
\usepackage{amsthm,amsfonts,amsmath,mathrsfs,amssymb}
\usepackage{dsfont}
\usepackage{mathtools}
\usepackage[T1]{fontenc}
\usepackage[utf8]{inputenc}
\usepackage{enumerate}
\usepackage{comment}
\usepackage[left=4cm,top=4cm,right=4cm,bottom=4cm]{geometry}
\numberwithin{equation}{section}

\usepackage{geometry}
\usepackage{graphicx}
\usepackage{amssymb}
\usepackage{amsmath}
\usepackage{amsthm}
\usepackage{listings}
\usepackage[colorlinks=true,urlcolor=blue,citecolor=blue,linkcolor=blue]{hyperref}
\usepackage{esint}
\usepackage{xcolor}

\usepackage{etex}
\usepackage[all]{xy}
\usepackage{pgf,tikz}
\usetikzlibrary{arrows}
\usepackage{pgfplots,pgfcalendar}
\usetikzlibrary{patterns}
\usepackage{hyperref}

\def\D{{\mathbb D}}  \def\T{{\mathbb T}}
\def\C{{\mathbb C}}  \def\N{{\mathbb N}}
\def\Z{{\mathbb Z}}

\def\R{\mathbb R}
\def\Q{\mathbb Q}
\def\({\left(}       \def\){\right)}

\def\Re{{\sf Re}\,}

\newtheorem{theorem}{Theorem}[section]
\newtheorem{lemma}[theorem]{Lemma}
\newtheorem{proposition}[theorem]{Proposition}
\newtheorem{corollary}[theorem]{Corollary}

\theoremstyle{definition}
\newtheorem{definition}[theorem]{Definition}
\newtheorem{example}[theorem]{Example}

\theoremstyle{remark}
\newtheorem{remark}[theorem]{Remark}

\numberwithin{equation}{section}

\makeatletter
\@namedef{subjclassname@2020}{%
  \textup{2020} Mathematics Subject Classification}
\makeatother

\begin{document}
\title[Semigroups of composition operators]{Semigroups of composition operators on Hardy spaces of Dirichlet series}

\author[M.D. Contreras]{Manuel D. Contreras}
\address{Departamento de Matem\'atica Aplicada II and IMUS, Escuela T\'ecnica Superior de Ingenier\'ia, Universidad de Sevilla,
	Camino de los Descubrimientos, s/n 41092, Sevilla, Spain}
\email{contreras@us.es}

\author[C. G\'omez-Cabello ]{Carlos G\'omez-Cabello}
\address{Departamento de Matem\'atica Aplicada II and IMUS, Edificio Celestino Mutis, Avda. Reina Mercedes, s/n. 41012 - Sevilla, Universidad de Sevilla,
Sevilla, Spain}
\email{cgcabello@us.es}

\author[L. Rodr\'iguez-Piazza]{Luis Rodr\'iguez-Piazza}
\address{Departmento de An\'alisis Matem\'atico and IMUS, Facultad de Matem\'aticas, Universidad
	de Sevilla, Calle Tarfia, s/n 41012 Sevilla, Spain}
\email{piazza@us.es}


\date{\today}


\thanks{This research was supported in part by Ministerio de Econom\'{\i}a y Competitividad, Spain,  and the European Union (FEDER), project PGC2018-094215-13-100,  and Junta de Andaluc{\'i}a, FQM133 and FQM-104.}

\maketitle

\begin{abstract}
We consider continuous semigroups of analytic functions $\{\Phi_t\}_{t\geq0}$ in the so-called \emph{Gordon-Hedenmalm class}  $\mathcal{G}$, that is, the family of analytic functions $\Phi:\C_+\to\C_+$ giving rise to bounded composition operators in the Hardy space of Dirichlet series $\mathcal{H}^2$. We show that there is a one-to-one correspondence between continuous semigroups $\{\Phi_{t}\}_{t\geq0}$ in the class $\mathcal G$  and  strongly continuous semigroups of composition operators $\{T_t\}_{t\geq0}$, where $T_t(f)=f\circ\Phi_t$, $f\in\mathcal{H}^2$. We extend these results for the range $p\in[1,\infty)$. For the case $p=\infty$, we prove that there is no non-trivial strongly continuous semigroup of composition operators in $\mathcal{H}^\infty$. We characterize the infinitesimal generators of continuous semigroups in the class $\mathcal G$
as  those Dirichlet series sending $\mathbb C_{+}$ into its closure.  Some dynamical properties of the semigroups are obtained from a description of the Koenigs map of the semigroup. 

\end{abstract}

\tableofcontents

\section{Introduction}

The study of Hardy spaces of Dirichlet series is a topic of increasing interest in the literature in the last 25 years starting with the fundamental paper of  Hedenmalm, Lindqvist, and  Seip \cite{HedLinSeip}.  The two recent monographs \cite{peris} and \cite{queffelecs} are samples of this reality where the state of the art can be found. Among the different relevant topics regarding these spaces, the study of composition operators plays a significant role. After some considerable effort, Gordon and Hedenmalm \cite{gorhed} got a characterization of bounded composition operators on the Hardy space of Dirichlet series $\mathcal H^{2}$. They introduced the following class where, as usual, given $\theta\in\R$, we denote
$
\C_{\theta}:=\{s\in\C:\text{Re}(s)>\theta\}
$
and  $\C_{+}=\C_{0}$: 
\begin{definition} Given an analytic function $\Phi:\C_{+}\to\C_{+}$, we say that $\Phi$ belongs to the {\sl Gordon-Hedenmalm class} $\mathcal{G}$ if:
\begin{enumerate}
    \item There exists $c_\Phi\in\N\cup\{0\}$ and $\varphi$ a Dirichlet series  such that
    \begin{equation}\label{chosenone}
        \Phi(s)=c_{\Phi}s+\varphi(s), \quad s\in \C_{+}.
    \end{equation}
    \item  If $c_{\Phi}=0$, then $\Phi(\C_+)\subset \C_{1/2}$.
\end{enumerate}
The value $c_{\Phi}$ is known as the {\sl characteristic} of the function $\Phi$. 
\end{definition}
Gordon and Hedenmalm proved that a composition operator is bounded on $\mathcal H^{2} $ if and only if its symbol belongs to $\mathcal G$. Later on, Bayart \cite{bayarto} obtained some results on the boundedness of composition operators on the spaces $\mathcal H^{p}$ for $1\leq p\leq +\infty$, being an open problem the complete characterization for $p\notin 2\N \cup\{ +\infty\}$. It is already known that $\Phi \in  \mathcal{G}$  is not sufficient to assure the boundedness of the composition operator $C_{\Phi}$ on $\mathcal H^{p}$, for $1\leq p<2$. This is a consequence of a deep result due to Harper \cite{Harper} about the so-called local embedding problem and  the previous work of Bayart and Brevig \cite{Bayart-Brevig} (see also \cite[p. 274]{queffelecs}).

Our goal is  to study semigroups of composition operators on the Hardy space of Dirichlet series $\mathcal H^{p}$ for $1\leq p\leq +\infty$. The study of semigroups of composition operators in the setting of Banach spaces of analytic functions was initiated with the seminal paper of Berkson and Porta \cite{porta}. They proved that a semigroup of composition operators $\{C_{\Phi_{t}}\}$ is strongly continuous in the Hardy space of the unit disc if and only if the family of functions  $\{\Phi_{t}\}$ is a continuous semigroup of holomorphic functions in the unit disc. This research has been extended to other spaces, having an analogue characterization in cases like Bergman spaces but rather different in others cases like the disc algebra, the Bloch space or BMOA. See \cite{Anderson}, \cite{Arevalo-Contreras-Piazza}, \cite{Avicou}, \cite{Betsakos}, \cite{BCDMPS}, \cite{ChaPar}, \cite{Gallardo} and references therein. 

In this paper, we tackle for the first time the study of semigroups of composition operators in the setting of Banach spaces of Dirichlet series. Before going to such topic, in Section \ref{sec:semigroups} we prove some properties of continuous semigroups of holomorphic functions (see Definition \ref{def:semigroup}) in the class $\mathcal G$.  In particular, we show that for functions belonging to such semigroups their characteristic must be 1, what implies that their Denjoy-Wolff point is $\infty$ and they are parabolic maps. 

In Section \ref{sec:semigroupscompositiooperators}, we prove 
\begin{theorem}[Theorems \ref{gordo1} and \ref{gordo11}]
Let $1\leq p<\infty$. Let $\{\Phi_t\}_{t\geq0}$ be a semigroup of analytic functions, such that $\Phi_t\in\mathcal{G}$ for every $t>0$ and denote by $T_{t}$ the composition operator $T_t(f)=f\circ\Phi_t$. Then, the following assertions are equivalent:
\begin{enumerate}[a)]
    \item $\{T_t\}_{t\geq0}$ is a strongly continuous semigroup in $\mathcal{H}^p$.
    \item $\{\Phi_t\}_{t\geq0}$ is a continuous semigroup.
    \item $\Phi_t(s)\to s$, as $t$ goes to $0$, uniformly in $\C_{\varepsilon}$, for every $\varepsilon>0$.
\end{enumerate}
\end{theorem}
We stress in this result that the convergence of $\Phi_t$ to the identity is  uniform in half-planes $\C_{\varepsilon}$ and not only on compact sets.

The main result in Section \ref{sec:infgenerator}   is a characterization of the holomorphic functions in the right half-plane that are infinitesimal generators of semigroups in the class $\mathcal G$. Namely, 
\begin{theorem}[Theorem \ref{gcoro}]\label{gorgointroduction}
Let $H:\C_+\to\overline{\C}_+$ be analytic. Then, the following statements are equivalent:
\begin{enumerate}[a)]
    \item $H$ is the infinitesimal generator of a continuous semigroup of elements in the class $\mathcal{G}$.
    \item $H\in\mathcal{H}^{\infty}(\C_{\varepsilon}),$ for all $\varepsilon>0$.
    \item $H$ is a Dirichlet series.
\end{enumerate}
\end{theorem}

Beyond the semigroups of translations ($\Phi_{t}(s)=s+at$, with $a\in \overline{\C_{+}}\setminus \{0\}$, for $s\in \C_{+}$ and $t\geq 0$), it is not easy to provide explicit examples of continuous semigroups in $\mathcal G$. One of the main interest of Theorem \ref{gorgointroduction} is that it guarantees the existence of many different examples.

We also use Theorem \ref{gorgointroduction}  to describe the infinitesimal generator of the semigroups of composition operators as well as its domain (see Proposition \ref{Prop:infgen}). At the end of this section, we prove that there is no non-trivial uniformly continuous semigroup of composition operators in $\mathcal H^{p}$ for $1\leq p<\infty$ (see Theorem \ref{Thm:uniformly}). In Section \ref{sec:Koenigs} we study the Koenigs function of a continuous semigroup in the class $\mathcal G$ and show that, up to the semigroups of automorphims,  all of them are maps of zero hyperbolic step (see Proposition \ref{Prop:zerohypstep}).

We conclude the paper with  Section \ref{sec:hinfinito} where we show that there is no non-trivial strongly continuous semigroups of compositions operators in $\mathcal H^{\infty}$ (see Theorem \ref{gordoinfty}).

In Section \ref{sec:hardy}, we collect some definitions and properties about Hardy spaces of Dirichlet series and composition operators between them. Most of these properties are well-known and we give a reference for them. In other cases, such properties are probably new or we could not find a reference. In such a case, for the sake of completeness, we have provided a proof.

\section{The $\mathcal{H}^p$ spaces and their composition operators} \label{sec:hardy}

\subsection{Dirichlet series}
We denote by $\mathcal{D}$ the space of convergent Dirichlet series, namely the series
$$
\varphi(s)=\sum_{n=1}^{\infty}a_nn^{-s},
$$
which are convergent  in some half-plane $\C_{\theta}$. To any Dirichlet series $\varphi(s)=\sum_{n=1}^{\infty}a_nn^{-s}$, one can associate the following abscissae:
$$
\sigma_{c}(\varphi)=\inf\{ \Re s:  \sum_{n=1}^{\infty}a_nn^{-s} \textrm{ is convergent}\};
$$
$$
\sigma_{b}(\varphi)=\inf\{ \sigma:  \sum_{n=1}^{\infty}a_nn^{-s} \textrm{ is bounded on } \C_{\sigma}\};
$$
$$
\sigma_{u}(\varphi)=\inf\{ \sigma:  \sum_{n=1}^{\infty}a_nn^{-s} \textrm{ is uniformly convergent on } \C_{\sigma}\};
$$
$$
\sigma_{a}(\varphi)=\inf\{ \Re s:  \sum_{n=1}^{\infty}a_nn^{-s} \textrm{ is absolutely convergent}\}.
$$
It is known that
$$
\sigma _{c}(\varphi)\leq \sigma _{b}(\varphi)=\sigma _{u}(\varphi)\leq \sigma _{a}(\varphi)\leq \sigma _{c}(\varphi)+1.
$$
The proofs of these inequalities can be found in \cite[Section 4.2]{queffelecs} or \cite[Chapter 1]{peris}.

We shall also consider certain subspaces of $\mathcal D$. Indeed, given $\Lambda$ a multiplicative semigroup of $\N$ (that is, $1\in \Lambda$ and $mn\in \Lambda $ whenever $m,n\in \Lambda$), we define
\begin{equation*}
    \mathcal{D}_{\Lambda}:=\left\{
    f\in\mathcal{D}: a_n=0\text{  if $n\not \in\Lambda$ } \right\}.
\end{equation*}
Note that the class $\mathcal{D}$ is nothing but the result of taking $\Lambda=\N$ in $\mathcal{D}_{\Lambda}$.

At some point  we will also need the following results concerning Dirichlet series. See \cite[Theorem 8.4.1]{queffelecs} for a proof of the first one and \cite[Remark 1.20]{peris} for a detailed proof of the second one.
\begin{theorem}\label{queffelecs}
Let $\sigma,\nu\in\R$. Consider $\varphi:\C_{\sigma}\to\C_{\nu}$ analytic such that it can be written as a  Dirichlet series in a certain half-plane. Then, $\sigma_u(\varphi)\leq\sigma $. In particular, $\varphi$ is bounded in $\C_{\sigma+\varepsilon}$ for all $\varepsilon>0$.
\end{theorem}

\begin{lemma}\label{lemaso}
Let $m$ be a positive integer and $\varphi(s)=\sum_{n=m}^{\infty}a_nn^{-s}$ be a Dirichlet series in $\mathcal{D}$ whose first non-zero coefficient is $a_m$. Then, $m^s\varphi (s)\to a_m$ as $\Re (s)\to+\infty$.
\end{lemma}

The next theorem is probably well-known for specialists, but we could not find a reference so that we have included its proof. 

\begin{theorem}\label{thm:nounivalent}
Let $\varphi(s)=\sum_{n=1}^{\infty}a_nn^{-s}$ be a Dirichlet series convergent in $\C_{\eta}$. Then $\varphi$ is not one-to-one in  $\C_{\eta}$.
\end{theorem}
\begin{proof}
Take $N=\min\{n\geq 2:\, a_n\neq 0\}$. We may  assume that $a_1=0$ and $a_N=1$. Write $\varphi (s)=N^{-s}+h(s)$ and $\gamma(s)=N^{-s}$ for all $s\in \C_\eta$. 

Take $\varepsilon< (N-1)/N^2$. Since $\lim_{\Re s\to +\infty} h(s)N^{s}=0$ (see Lemma \ref{lemaso}), there is $s_0>\eta$ such that $|h(s)|<N^{-\Re s }\varepsilon $ whenever  $\Re s>s_0$. 

Take $x>s_0+1$  and $m\in \Z$. Write
$g(s):= N^{-s}- N^{-x}$ and $f(s):=\varphi(s)- N^{-x}$, for $s\in \C_\eta$. Consider the segments
\begin{equation*}
\begin{split}
\Gamma_{a,m} &=[a-i\pi/\ln (N), a+i\pi/\ln (N)] +2m\pi i/\ln(N),\\
 \Delta_{a,m}&=[a-i\pi/\ln (N), a+2-i\pi/\ln (N)]+2m\pi i/\ln(N),
\end{split}
\end{equation*}
for every $a\in \R$.
Notice that $\gamma(\Gamma_{a,m})$ is the circle centered at $0$ and with radius $N^{-a}$, that is $C(0,N^{-a})$, and $\gamma(\Delta_{a,m})$ is the segment that joins the points $-N^{-a}$ with $- N^{-(a+2)}$. Now, take the rectangle
$$
\Gamma_m=\Gamma_{x-1,m}\cup\Gamma_{x+1,m}\cup \Delta_{x-1,m}\cup \Delta_{x-1,m+1}.
$$
On the one hand, notice that $\Gamma_m$ is contained in the half-plane $\Re s>s_0$. Thus, for any $s\in \Gamma_{m}$, it holds 
\begin{equation}\label{eq:nounivalent1}
|g(s)-f(s)|=|h(s)|<N^{-\Re s }\varepsilon\leq N^{-(x-1)}\varepsilon.
\end{equation}
On the other hand, observe that
$$
\gamma (\Gamma_m)=\gamma (\Gamma_0)=C(0,N^{-(x-1)})\cup C(0,N^{-(x+1)})\cup [ - N^{-(x-1)}, - N^{-(x+1)}].
$$
Thus, for  $s\in \Gamma_m$, we have 
\begin{equation}\label{eq:nounivalent2}
\begin{split}
|g(s)|&=|\gamma(s)- N^{-x}|=| N^{-s}-N^{-x}|\\
&\geq \min\{| N^{-(x-1)}-N^{-x}|,| N^{-(x+1)}-N^{-x}|\} \geq N^{-x-1} (N-1).
\end{split}
\end{equation}
Since $\varepsilon< (N-1)/N^2$, we deduce that for all $s\in \Gamma_m$, \eqref{eq:nounivalent1} and  \eqref{eq:nounivalent2} imply
$$
|g(s)-f(s)|<|g(s)|.
$$
The point $x+i2m\pi/\ln(N)$ is the center of $\Gamma_m$ and $g(x+i2m\pi/\ln(N))=0$ (in fact, it is the unique zero of $g$ in the interior of such rectangle). Now, Rouch\'e's Theorem implies that the equation $\varphi  (s)=N^{-x}$ has a solution in the interior of each $\Gamma_{m}$ what clearly shows that $\varphi$ is not univalent. 
\end{proof}

\subsection{The $\mathcal{H}^p$ spaces}
Before moving onto the study of composition operators between the $\mathcal{H}^p$ spaces, we shall recall how these ones are constructed.

If $\{ p_{j}\}_{j\geq 1}$ is the increasing sequence of prime numbers, given a natural number $n$, we will denote its prime number factorization
$$
n=p_{1}^{\kappa_{1}}p_{2}^{\kappa_{2}} \cdots p_{d}^{\kappa_{d}}
$$
by $n=(p_{j})^{\kappa}$. This associates uniquely to $n$ the finite multi-index $\kappa (n)=(\kappa_{1}, \kappa_{2}, \dots).$ 

The Bohr lift of a Dirichlet series 
$$
\varphi (s)=\sum_{n\geq 1}a_{n} n^{-s}
$$
is the power series $\mathcal B \varphi (z)=\sum _{n\geq 1}a_{n}z^{\kappa(n)}$, $z\in \D^{\infty}$. For $1\leq p<\infty$, we define the Hardy space of Dirichlet series $\mathcal H ^{p}$ as the space of Dirichlet series $\varphi$ such that $\mathcal B\varphi$ is in $H^{p }(\mathbb D^{\infty})$, and we set
$$
\|\varphi \|_{\mathcal H ^{p}}:=\|\mathcal B \varphi \|_{H^{p }(\mathbb D^{\infty})}=\left(\int_{\mathbb T ^{\infty}}|\mathcal B \varphi(z)|^{p}\, dm_{\infty}(z)\right)^{1/p},
$$
where $m_{\infty}$ denotes the Haar measure of infinite polytorus $\mathbb T ^{\infty}$, which is simply the product of the normalized Lebesgue measure of the torus $\mathbb T$ in each variable. For $p=2$, we have that
$$
\|\varphi \|_{\mathcal H ^{2}}=\left( \sum_{n=1}^{\infty}|a_{n}|^{2} \right)^{1/2}.
$$
 It is easy to see that  $\mathcal{H}^q\subset \mathcal{H}^p$ and that given $\varphi\in\mathcal{H}^q$ it holds $\| \varphi \|_p\leq\| \varphi \|_q$,  whenever $1\leq p\leq q< \infty$. 
The construction of the Hardy spaces of Dirichlet series that we have just outlined  gives rise to a Banach space of analytic functions in the half-plane $\C_{1/2}$. In fact, their abscissae of absolute convergence, $\sigma_{a}(\varphi)$, is smaller or equal to $1/2$. This is a consequence of Helson's inequality: 
 \begin{equation}\label{estima}
  \left(
 \sum_{n=1}^{\infty}\frac{|a_n|^2}{d(n)}
 \right)^{\frac 12}\leq\|\varphi\|_1
    \end{equation}
 for any function $\varphi(s)=\sum_{n=1}^{\infty}a_nn^{-s}\in\mathcal{H}^1$, where $d(n)$ denotes the number of divisors of $n$ (see \cite[Theorem 6.5.9]{queffelecs}).

For  $\varepsilon\geq 0$, 
 the space of bounded Dirichlet series $\mathcal{H}^{\infty}(\C_{\varepsilon})$ consists of all analytic functions bounded in $\C_\varepsilon$ such that they can be written as a Dirichlet series in a certain half-plane. If we endow the space $\mathcal{H}^{\infty} (\C_{\varepsilon})$ with the norm given by
 \[
\|\varphi\|_{\mathcal{H}^{\infty}(\C_{\varepsilon})}=\sup_{s\in\C_\varepsilon}|\varphi(s)|,
\]
we obtain a Banach space. For $\varepsilon=0$, we simply write   $\mathcal{H}^{\infty}$ to denote $\mathcal {H}^{\infty } (\C_{+})$. It holds that $\mathcal H^{\infty}\subset \mathcal H^{p}$ for all $p<+\infty$.

Given $\Lambda$ a multiplicative semigroup of $\N$, a related space that will appear throughout the paper is 
 $$
 \mathcal{H}^{\infty}_{\Lambda}(\C_{\varepsilon})=\{\varphi(s)=\sum _{n=1}^{\infty}a_{n}n^{-s}\in  \mathcal{H}^{\infty}(\C_{\varepsilon}):\ a_n=0\text{  if \ } n\not \in\Lambda   \}.
$$
It is clear that $ \mathcal{H}^{\infty}_{\Lambda}(\C_{\varepsilon})$ is a closed subspace of $ \mathcal{H}^{\infty}(\C_{\varepsilon})$. Moreover,  the class $\mathcal{D}_{\Lambda}$ is an algebra and so $ \mathcal{H}^{\infty}_{\Lambda}(\C_{\varepsilon})$ is a unital Banach algebra. Indeed, let $\varphi_1(s)=\sum_{n=1}^{\infty}b_nn^{-s}$ and $\varphi _2(s)=\sum_{n=1}^{\infty}c_nn^{-s}$ be two elements in $\mathcal{D}_{\Lambda}$. Then,  their product is given by
\begin{align*}
    \varphi_{1}\varphi_2(s)=\sum_{n=1}^{\infty}\left(\sum_{kl=n}b_kc_l\right)n^{-s}=\sum_{n=1}^{\infty}d_nn^{-s}.
\end{align*}
Now, if $n=kl\not\in\Lambda$, we have that either $k\not\in\Lambda$ or $l\not\in\Lambda$. In any case, this implies that the corresponding coefficient vanishes and so does $d_n$. Then, $    \varphi_{1}\varphi_2\in\mathcal{D}_{\Lambda}$. Since $\mathcal{D}_{\Lambda}$ is a linear space, we deduce that it is also an algebra.

It is known that the pointwise evaluation functionals are bounded on $\mathcal{H}^p$. Namely, given $s\in \C$ such that $\Re s>1/2$, the functional $\delta_{s}(\varphi)=\varphi (s)$, for all $\varphi\in \mathcal{H}^p$, satisfies that
\begin{equation}\label{normevaluationfunciotnal}
\|\delta_{s}\|_{(\mathcal{H}^p)^{*}}=(\zeta(2\Re s))^{1/p},
\end{equation}
where $\zeta$ denotes  the Riemann zeta map (see \cite[page 273]{queffelecs}).   In the case of $\mathcal H^{2}$, we will also need  the boundedness of the evaluation of the derivative. Thus, for the sake of completeness, we have included a brief proof of this  fact. 

\begin{lemma}\label{pointwiseevaluaction}
Let $k\in \N\cup\{0\}$ and $s\in \C_{1/2}$. The functional $\delta^k_s:\mathcal{H}^2\to\C$ given by
\[
\varphi \mapsto \varphi ^{(k)}(s)
\]
is bounded. In fact, given $\sigma>0$, it is uniformly bounded  for $s\in\C_{1/2+\sigma}$.
\end{lemma}
\begin{proof}
Writing $\varphi (s)=\sum_{n=1}^{\infty}a_nn^{-s}$, we recall that
\[
\varphi^{(k)}(s)=\sum_{n=1}^{\infty}(-\log n)^ka_nn^{-s},\quad s\in \C_{1/2}.
\]
Fix $\sigma>0$. For $s\in\C_{1/2+\sigma}$,
$$
|\varphi^{(k)}(s)  |\leq \sum_{n=1}^{\infty}\frac{(\log n)^k|a_n|}{n^{\text{Re}(s)}}\leq\left(\sum_{n=1}^{\infty}\frac{(\log n)^{2k}}{n^{1+2\sigma}}\right)^{1/2}\left(\sum_{n=1}^{\infty}|a_n|^2\right)^{1/2}<\infty.
$$
Then, $|\varphi^{(k)}(s)|\leq C(\sigma)\|\varphi \|_{\mathcal{H}^2}$, $s\in\C_{1/2+\sigma}$. 
\end{proof}

The next Montel-type theorem for Dirichlet series,  due to Bayart, will be needed several times in this paper. We state it for the sake of clearness:
 \begin{theorem}[Bayart  \cite{bayarto}] \label{bayarto}
Let $\{f_j\}_j$ be a bounded sequence in $\mathcal{H}^{\infty}$. Then, there exist both a subsequence $\{f_{n_k}\}$  and a function $f\in\mathcal{H}^{\infty}$, such that $f_{n_k}$ converges uniformly to $f$ on each half-plane $\C_{\varepsilon}$, for all $\varepsilon>0$.
\end{theorem}

\subsection{Composition operators}
In their seminal work \cite{gorhed}, Gordon and Hedenmalm focused their attention on composition operators in the Hilbert space $\mathcal{H}^2$. More precisely, given an analytic function $\Phi:\C_{1/2}\to\C_{1/2}$, the composition operator  is defined as $C_{\Phi}(f)=f\circ\Phi$, whenever $f$ is analytic in $\C_{1/2}$. 
As a matter of nomenclature, $\Phi$ is often said to be the \emph{symbol} of the composition operator $C_{\Phi}$. Now, what are the minimal requirements for a given analytic function $\Phi:\C_{1/2}\to\C_{1/2}$ to define a bounded composition operator in $\mathcal{H}^2$? Reciprocally, once we have the boundedness of such an operator, what properties does $\Phi$ satisfy? 
To answer this matter, in \cite{gorhed} the authors introduced the now so-called \emph{Gordon-Hedenmalm class}, denoted by $\mathcal{G}$, even though they did not use this name in their original paper. This class is one of the cornerstones of the present work. 

\begin{definition} Let $\Phi:\C_{+}\to\C_{+}$ be an analytic function.
\begin{enumerate}
\item We say that $\Phi$ belongs to the class $\mathcal{G}_{\infty}$ if there exist $c_\Phi\in\N\cup\{0\}$ and $\varphi\in\mathcal{D}$ such that
    \begin{equation}\label{chosenone}
        \Phi(s)=c_{\Phi}s+\varphi(s).
    \end{equation}
The value $c_{\Phi}$ is known as the {\sl characteristic} of the function $\Phi$. 
\item We say that $\Phi$ belongs to the {\sl Gordon-Hedenmalm class} $\mathcal{G}$ if $\Phi\in  \mathcal{G}_{\infty}$ and  $\Phi(\C_+)\subset \C_{1/2}$ in case $c_{\Phi}=0$.
\end{enumerate}
\end{definition}
\begin{remark} Assume $c_{\Phi}\in \N$ for a certain $\Phi \in \mathcal G_{\infty}$ like in \eqref{chosenone}. By Lemma \ref{lemaso}, it holds 
$$
c_{\Phi}=\lim_{\Re s\to +\infty} \frac{\Phi(s)}{s}.
$$
Thus, by Julia-Wolff-Carath\'eodory's Lemma (see \cite[Theorem 1.7.8]{manoloal}), $\Re \Phi(s)\geq c_{\Phi}\Re s$ for all $s\in \C_{+}$. Therefore, $\varphi$ sends $\C_{+}$ into $\overline {\C_{+}}$. In addition, by Theorem \ref{queffelecs}, we have that  $\sigma_u(\varphi)\leq0$. 
\end{remark}

The result Gordon and Hedenmalm proved reads as follows. The reader is addressed to \cite{gorhed} for a proof.
\begin{theorem}[Gordon-Hedenmalm]\label{gordon}
An analytic function $\Phi:\C_{\frac12}\to\C_{\frac12}$ defines a bounded composition operator $\mathcal{C}_{\Phi}:\mathcal{H}^2\to\mathcal{H}^2$ if and only if $\Phi$ has a holomorphic extension to $\C_{+}$ that  belongs to the class $\mathcal{G}$. 
\end{theorem}

This characterization is also valid for the spaces $\mathcal H^{p}$ whenever $p\in 2\N$.  
As far as we know, the characterization of the boundedness of composition operators in the other Hardy spaces of Dirichlet series remains open. We use the following  result from Bayart's work \cite{bayarto} (see also \cite[Theorem 8.10.11]{queffelecs}).
\begin{theorem}[Bayart]\label{hp}
Let $\Phi:\C_{1/2}\to\C_{1/2}$ be analytic and $1\leq p<\infty$.
\begin{enumerate}[(a)]
    \item If $C_{\Phi}:\mathcal{H}^p\to\mathcal{H}^p$ is bounded, then  $\Phi\in\mathcal{G}$.
    \item If $\Phi\in\mathcal{G}$ and  $c_{\Phi}\geq1$, then $C_{\Phi}:\mathcal{H}^p\to\mathcal{H}^p$ is bounded.
\end{enumerate}
\end{theorem}
\begin{remark}\label{norm}
From the proof of the sufficiency for the non-zero characteristic case in \cite[Theorem B]{gorhed}, it follows that $\|C_{\Phi}\|_{\mathcal H^{2}}\leq1$. That is, the composition operator $C_{\Phi}$ is a contraction on $\mathcal H^{2}$ whenever the characteristic $c_{\Phi}$ is a natural number. In fact, this was extended to the other values of $p\in [1,\infty)$ by Bayart (see \cite[Theorem 8.10.1]{queffelecs}). This observation will play a key role in subsequent sections.
\end{remark}

The characterization of bounded composition operators on $\mathcal{H}^\infty$  was obtained by Bayart in \cite{bayarto} (see also \cite[Proposition 2]{bayart-castillo}). It shows that there are more bounded composition operators on  $\mathcal{H}^\infty$ than on $\mathcal{H}^2$:
\begin{theorem}\label{boundedness-Hinfty}
A function $\Phi:\C_{+}\to\C_{+}$ defines a bounded composition operator $\mathcal{C}_{\Phi}:\mathcal{H}^\infty\to\mathcal{H}^\infty$ if and only if $\Phi$ belongs to the class $\mathcal G_{\infty}$.
\end{theorem}

Next result confirms the well behaviour of the class $\mathcal{D}_{\Lambda}$ under the action of $C_{\Phi}$, where $\Phi(s)=c_{\Phi}s+\varphi(s)$ with $\varphi\in\mathcal{D}_{\Lambda}$ and $c_{\Phi}\in\N\cup\{0\}$. Our proof is a slight modification of the  proof of \cite[Theorem A]{gorhed} once we know that $\mathcal{D}_{\Lambda}$ is an algebra.

\begin{theorem}\label{menudeo}
Let $\theta,\nu\in\R$. Consider $\Phi:\C_{\theta}\to\C_{\nu}$ an analytic map such that it can be written as
\begin{equation*}
    \Phi(s)=c_{\Phi}s+\varphi(s),
\end{equation*}
where $c_{\Phi}\in\N\cup\{0\}$ and $\varphi\in\mathcal{D}_{\Lambda}$. Then, $\Phi$ generates a composition operator $\mathcal{C}_{\Phi}:\mathcal{D}_{\Lambda}\to\mathcal{D}_{\Lambda}$, that is, $f\circ \Phi\in \mathcal{D}_{\Lambda}$ for all $f\in \mathcal{D}_{\Lambda}$.
\end{theorem}
\begin{proof}
Let $f(s)=\sum_{n\geq1}a_nn^{-s}\in\mathcal{D}_{\Lambda}$. Then,
\begin{align*}
f\circ\Phi(s)  =  \sum_{n=1}^{\infty}a_nn^{-\Phi(s)}=\sum_{n=1}^{\infty}a_nn^{-c_{\Phi}s-\varphi(s)}.
\end{align*}
Since $\varphi\in\mathcal{D}_{\Lambda}$, there exists $\sigma>0$ such that $\varphi\in\mathcal{H}_{\Lambda }^{\infty}(\C_{\sigma})$. Now, as $\mathcal{H}_{\Lambda}^{\infty}(\C_{\sigma})$ is a Banach algebra, we have that
\[
n^{-\varphi(s)}=\exp(-\log n\, \varphi(s))\in \mathcal{H}^{\infty}_{\Lambda}(\C_{\sigma}).
\]
On the other hand, for $n\in\Lambda$,  clearly, $n^{-c_{\Phi}s}\in\mathcal{D}_{\Lambda}$. Thus $n^{-c_{\Phi}s}\in\mathcal{H}^{\infty}_{\Lambda}(\C_{\sigma})$.
Putting all together, we conclude that
\[
n^{-c_{\Phi}s}n^{-\varphi(s)}\in\mathcal{H}^{\infty}_{\Lambda}(\C_{\sigma}).
\]
Taking this into account, for $\sigma_1>\sigma$ big enough so that $f$ converges absolutely in $\C_{\sigma_1}$, 
\begin{align*}
    \sup_{s\in\C_{\sigma_1}}|a_nn^{-c_{\Phi}s}n^{-\varphi(s)}|\leq 
    C|a_n|n^{-c_{\Phi}\sigma_1}.
    \end{align*}
This proves that the composition $f\circ\Phi$ lies in $\mathcal{H}^{\infty}_{\Lambda}(\C_{\sigma_{1}})$ and so in $\mathcal D_{\Lambda}.  $
\end{proof}

The classes $\mathcal G$ and $\mathcal G_{\infty}$ are stable under composition and the characteristic of the composition is the product of the characteristics. This fact is relevant when dealing with semigroups of functions in these classes. 

\begin{proposition}\label{colorario}
Let $\Phi,\Psi\in\mathcal{G}_{\infty}$ (resp. $\Phi,\Psi\in\mathcal{G}$). Then, $\Phi\circ\Psi\in\mathcal{G}_{\infty}$ (resp. $\Phi\circ\Psi\in\mathcal{G}$). Moreover, $c_{\Phi\circ\Psi}=c_{\Phi}c_{\Psi}$.
\end{proposition}
\begin{proof}
Given $\Phi,\Psi\in\mathcal{G}_{\infty}$ and $f\in\mathcal{H}^\infty$, Theorem \ref{boundedness-Hinfty} guarantees that $f\circ\Phi\in\mathcal{H}^\infty$. Now, once more, $f\circ(\Phi\circ\Psi)=(f\circ\Phi)\circ\Psi\in\mathcal{H}^\infty$. Then, $\Phi\circ\Psi$ is a symbol of a bounded composition operator from $\mathcal{H}^\infty$ into $\mathcal{H}^\infty$. Therefore, again by Theorem \ref{boundedness-Hinfty}, necessarily, $\Phi\circ\Psi\in\mathcal{G}_{\infty}$. Similarly, replacing the role in this argument of $\mathcal H^{\infty}$ by $\mathcal{H}^2$ and Theorem \ref{boundedness-Hinfty} by Theorem \ref{gordon}, we get that $\Phi\circ\Psi\in\mathcal{G}$ whenever $\Phi, \Psi\in\mathcal{G}$.

For the second part of the statement, since we already know that $\Phi\circ\Psi\in\mathcal{G}_{\infty}$, we can write
\begin{equation}\label{segundo}
\Phi\circ\Psi(s)=c_{\Phi\circ\Psi}s+\eta(s), \quad c_{\Phi\circ\Psi}\in\N\cup\{0\} \text{  and  }\eta\in\mathcal{D}.
  \end{equation}
  But also
  \begin{equation}\label{primero}
    \Phi\circ\Psi(s)=c_{\Phi}\Psi(s)+\varphi(\Psi(s))=c_{\Phi}c_{\Psi}s+c_{\Phi}\psi(s)+\varphi(\Psi(s)).
\end{equation}
Identifying \eqref{primero} and \eqref{segundo}, we find that
\begin{align*}
    0&=
        (c_{\Phi}c_{\Psi}-c_{\Phi\circ\Psi})s+c_{\Phi}\psi(s)+\varphi(\Psi(s))-\eta(s).
  \end{align*}
 Dividing by $s$ in the latter identity, we find that
  \begin{align*}
       0=c_{\Phi}c_{\Psi}-c_{\Phi\circ\Psi}+\frac1s\left(c_{\Phi}\psi(s)+\varphi(\Psi(s))-\eta(s)
    \right).
   \end{align*}
   If we  let $\text{Re}(s)\to+\infty$,  by Lemma \ref{lemaso}, the term in parenthesis converges. Hence, it is bounded. This yields the desired conclusion.
\end{proof}

\section{Semigroups of analytic functions}\label{sec:semigroups}
Semigroups of analytic functions in  the unit disc $\D$ (and then in the right half-plane) have been a subject of study since the early 1900s. In 1978, Berkson and Porta \cite{porta} studied continuous semigroups of analytic self-maps of the unit disc in connection with composition operators. This paper meant the resurgence of this area. The state of the art can be seen in \cite{manoloal}.
We recall the definition straightaway.
\begin{definition} \label{def:semigroup}
We say that a family $\{\Phi_t\}_{t\geq0}$ of analytic functions $\Phi_t:\C_{+}\to\C_{+}$ is a semigroup if it verifies the following two algebraic properties:
\begin{itemize}
\item $\Phi_0(s)=s.$
    \item For every $t,u\geq0$, $\Phi_t\circ\Phi_u(s)=\Phi_{t+u}(s)$.
    \end{itemize}
If, in addition, it satisfies that $\Phi_t$ converges to $\Phi_0$ uniformly on compact subsets of $\C_{+}$ as $t\to0^+$, we say that it is a continuous semigroup.
\end{definition}

It is worth recalling that any holomorphic function in a continuous semigroup is univalent (see, e.g., \cite[Theorem 8.1.17]{manoloal}).

A key goal in this paper is the study of semigroups of analytic functions $\{\Phi_{t}\}$ in the Gordon-Hedenmalm class $\mathcal{G}$ and in the class $\mathcal{G}_{\infty}$. From now on, for convenience, we shall write $c_t$ instead of $c_{\Phi_t}$. Our first result studies the behaviour of the mapping $t\mapsto c_t$. As we are about to see, the continuous semigroup structure forces this mapping to be necessarily constantly equal to $1$. This is the content of the next proposition.
\begin{proposition}\label{cste1}
Let $\{\Phi_t\}_{t\geq0}$ be a continuous semigroup of analytic functions in the class $\mathcal{G}_{\infty}$. Then, the sequence of symbols $\{c_t\}_{t\geq0}$ is constantly equal to $1$.
\end{proposition}
\begin{proof}
By Proposition \ref{colorario}, the characteristic of $\Phi_t\circ\Phi_u(s)=\Phi_{u+t}(s)$ satisfies $c_{u+t}=c_uc_t$. We claim that  the application $f:\R_+\cup\{0\}\to\R$ given by $f(t):=c_t$ is measurable. Indeed,  for each $n\in\N$, define the functions $h_n(t)=\Phi_t(n)/n$. By hypothesis, these functions are continuous as functions of $t$ (see \cite[Theorem 8.1.15]{manoloal}).  Lemma \ref{lemaso} yields
\[
\lim_{n\to\infty}h_n(t)=\lim_{n\to\infty}\frac{\Phi_t(n)}{n}=c_t.
\]
Recalling that the pointwise limit of continuous functions is measurable, 
the claim follows.
Summing up, $f$ is a measurable function satisfying that $f(u+t)=f(u)f(t)$, for all $s,t\geq 0$. 
By  \cite[Theorem 8.1.14]{manoloal}, either $f(t)\in\{0,1\}$, or there exists $\lambda\in\R$ such that $f(t)=e^{\lambda t}$ for every $t\geq0$.
Since $c_t\in\N\cup \{0\}$ for every $t\geq0$, necessarily $c_{t}\in\{0,1\}$. Since the function $\Phi_{t}$ must be univalent, by Theorem \ref{thm:nounivalent}, we have that $c_{t}\geq 1$. Consequently, $c_t=1$ for every $t\geq0$. 
\end{proof}

\begin{remark}
If the functions of the semigroup belong to the class $\mathcal{G}$, there is an alternative way to conclude the above proof without using the univalence of the functions of the semigroup. Indeed, take $s_0>0$ such that $c_{s_0}=0$. Then $c_t=c_{t-s_{0}}c_{s_{0}}=0$ for every $t\geq s_0$. Note that this implies that whenever $c_u=1$, then $c_t=1$  for every $t\leq u$.  Therefore, there is a point $t_0$ such that $c_t=1$ if $t<t_0$ and $c_t=0$ if $t>t_0$.  We claim that $t_0=\infty$. If it were equal to zero, it would mean that $c_t=0$ for every $t>0$. Now, since $\Phi_t\to \Phi_0$ as $t\to0$ on compact sets of $\C_+$, we would have that $\Phi_t(1/4)\to 1/4$ as $t\to0$. However, since $\Phi_t\in\mathcal{G} $, we know that $\text{Re}(\Phi_t(1/4))>1/2$, in contradiction with our last statement.  Assume that $0<t_0<\infty$. Then 
\[
c_{\frac43t_0}=c_{\frac23t_0}c_{\frac23t_0}=1.
\]
However, this contradicts the fact that $c_t=0$ for any $t>t_0$. Hence, $t_0=\infty$ and, consequently, $c_t=1$ for any $t\geq0$. 
\end{remark}

It is worth pointing out that there are non-continuous semigroups of functions in the class $\mathcal G$. For example, take a non-continuous function $f:\R\to \R$ such that $f(t+u)=f(t)+f(u)$ for all $t,u\in \R$ and consider $\Phi_{t}(s)=s+if(t). $ Clearly, $\{\Phi_{t}\}$ is a non-continuous semigroup in $\mathcal G$.

If a holomorphic self-map $\Phi$ of $\C_{+}$ is not an elliptic automorphism and has a fixed point $s_{0}$ in $\C_{+}$, then its  iterates converge uniformly on compact subsets of $\C_{+}$ to the constant function $s_{0}$ (see \cite[Proposition 1.8.3]{manoloal}). If $\{\Phi_{t} \}$ is a semigroup in the class $\mathcal G_{\infty}$,  by Proposition \ref{cste1}, $\Re \Phi_{t}(s)\geq c_t\Re s=\Re s$ for all $s$. This clearly implies that $\Phi_{t}$ has no fixed point in $\C_{+}$ and, in fact, by \cite[Theorem 8.3.1]{manoloal}, the Denjoy-Wolff point of the semigroup is $\infty$, that is,
\begin{equation}\label{DW}
\lim_{t\rightarrow +\infty }\Phi_{t}(s)=\infty,\quad s\in\C_{+}.
\end{equation}
Moreover, for every $t>0$, by Lemma \ref{lemaso},
$$
\lim_{\Re s\to +\infty} \frac{\Phi_{t}(s)}{s}=1.
$$
With the standard classification of dynamics, this means that each function $\Phi_{t}$ is a parabolic self-map of $\C_{+}$ (see \cite[Section 1.8]{manoloal} for the classification in the setting of the unit disc).

\section{Semigroups of composition operators} \label{sec:semigroupscompositiooperators}
The theory of strongly continuous semigroups of bounded operators on Banach spaces has been a fruitful tool in a great number of areas in Analysis. Let us recall this  notion. 
\begin{definition}
Let $X$ be a Banach space and $\{T_{t}\}_{t\geq 0}$ a family of bounded operators from $X$ into itself. We say that  $\{T_{t}\}_{t\geq 0}$ is a semigroup if  it verifies the following two algebraic properties:
\begin{enumerate}
\item[(i)] $T_0=\mathrm{Id}$, where $\mathrm{Id}$ denotes the identity map on $X$;
    \item[(ii)] For every $t,u\geq0$, $T_t\circ T_u=T _{t+u}$.
    \end{enumerate}
If, in addition, it satisfies that 
\begin{enumerate}
\item[(iii)] $\lim_{t\to 0^{+}}T_{t}f=f$ for all $f\in X$
    \end{enumerate} we say that it is a strongly continuous semigroup (also known as $C_{0}$-semigroup).
\end{definition}

It is well-known that (iii) is equivalent to the fact that, for each $f\in X$, the map $[0,+\infty)\ni t\mapsto T_{f}f$ is continuous \cite[Page 3, Proposition 1.3]{engels}.

Clearly, if we have a semigroup of analytic functions $\{\Phi_{t}\}$, we obtain a semigroup of composition operators $\{C_{\Phi_{t}}\}$, whenever such composition operators are well-defined. The converse is also true because a composition operator completely determines its symbol:

\begin{lemma}\label{regato}
Let $\Phi,\Psi:\C_{\eta}\to\C_{\eta}$, with $\eta\geq 0$, be two analytic functions such that $k^{-\Phi}=k^{-\Psi}$ for every $k\in\N$. Then, $\Phi=\Psi$. 
\end{lemma}
\begin{proof}
By hypothesis, for each $k\in \N$, 
\begin{align*}
 \exp((\Phi(s)-\Psi(s))\log k))=1, \quad \Re s>\eta.
\end{align*}
This forces that there is $m_{k}\in \Z$ such that $ (\Phi(s)-\Psi(s))\log k=2m_{k}\pi i$ for all $s$. If  $m_{2}$ is not zero, then $\log 3/\log2 \in\Q$. However, this is not true, so $m_{2}=0$ implying $\Phi=\Psi$.
\end{proof}
\begin{proposition}\label{Prop:algebraicsemigroup}
Let $1\leq p\leq\infty$. Let $\{\Phi_t\}_{t\geq0}$ be a collection of analytic functions in the class $\mathcal G_{\infty}$. If $\{T_t\}_{t\geq0}$ is a semigroup of composition operators in $\mathcal{H}^p$, where $T_t=f\circ\Phi_t$, then $\{\Phi_t\}_{t\geq0}$ is a semigroup of analytic functions.
\end{proposition}
\begin{proof}
The semigroup structure of $\{T_t\}_{t\geq0}$ guarantees that for every $f\in\mathcal{H}^p$
\[
f\circ\Phi_{t+u}=T_{t+u}(f)=T_t\circ T_u(f)=(f\circ\Phi_u)\circ\Phi_t=f\circ(\Phi_u\circ\Phi_t)
\]
for every $u,t\geq0$. By Lemma \ref{regato}, we have that $\Phi_{t+u}=\Phi_u\circ\Phi_t$ for every $u,t\geq0$. The same argument shows that $\Phi_{0}$ is the identity map.
\end{proof}

Now, we are ready to state and prove the main result of this section. It establishes   a one-to-one relationship between strongly continuous semigroups of composition operators in $\mathcal{H}^2$ and continuous semigroups in the Gordon-Hedenmalm class.

\begin{theorem}\label{gordo1}
Let $\{\Phi_t\}_{t\geq0}$ be a semigroup of analytic functions, such that $\Phi_t\in\mathcal{G}$ for every $t>0$ and denote by $T_{t}$ the composition operator $T_t(f)=f\circ\Phi_t$. Then, the following assertions are equivalent:
\begin{enumerate}[a)]
    \item $\{T_t\}_{t\geq0}$ is a strongly continuous semigroup in $\mathcal{H}^2$.
    \item $\{\Phi_t\}_{t\geq0}$ is a continuous semigroup.
    \item $\Phi_t(s)\to s$, as $t$ goes to $0$,  uniformly in $\C_{\varepsilon}$, for every $\varepsilon>0$.
\end{enumerate}
\end{theorem}

\begin{proof}
We begin showing $a)$ implies $b)$. Take $f(s)=m^{-s}$, with $m\in \N$, $m\geq 2$. Then, we consider the family of functions $g_t(s)=f\circ\Phi_t(s)=m^{-\Phi_t(s)}$. By $a)$, $g_{t}$ converges to $f$ in $\mathcal H^{2}$, as $t$ goes to $0$. The functions $g_t$ are analytic in $\C_+$ and for every $s\in\C_{+}$ and  $t\geq0$
\[
\left|
m^{-\Phi_t(s)}
\right|= m^{-\text{Re}(\Phi_t(s))}<1.
\]
Hence, for every $t\geq0$, $g_t(s)=m^{-\Phi_t(s)}\in H^{\infty}(\C_+)$ and, as we already know (see Theorem \ref{menudeo}), it is also in the class $\mathcal{D}$. In other words, for every $t\geq0$, $g_t\in\mathcal{H}^{\infty}$ and the family $\{g_t: \ t\geq 0\}$ is uniformly bounded in $\C_+$. 
Thus, Bayart's Theorem \ref{bayarto} guarantees that $m^{-\Phi_t(s)}\to m^{-s}$ uniformly on half-planes $\C_{\varepsilon}$ as $t\to 0$. 
 Multiplying by $m^s$, we have that
\begin{equation*}
m^{s-\Phi_t(s)}\to1, \quad t\to 0
  \end{equation*}
where the  convergence takes place uniformly in vertical strips of $\C_{\varepsilon}$, since the function $m^s$ is no longer bounded in any half-plane $\C_{\varepsilon}$. 

Fix $0<\varepsilon<\eta$ and $\mathbb S =\mathbb S(\varepsilon, \eta)=\C_{\varepsilon}\setminus \overline {\C_{\eta}}$. We assume that $\varepsilon+\frac{2\pi}{\ln 2}<\eta $. There is $\delta_{1}>0$ such that $|2^{s-\Phi_{t}(s)}-1|<1/2$ for all $s\in \mathbb S$ and $t<\delta_{1}$.  For each $t<\delta_{1}$, the image of the continuous function $\mathbb S\ni s\mapsto s-\Phi_{t}(s)$ must be contained in one of the connected components of the set $\{w\in \C: |2^{w}-1|<1/2\}$. Therefore there is $k_{t}\in \Z$ such that 
\begin{equation}\label{Eq:gordo1}
\left| s-\Phi_{t} (s)-\frac{2\pi i k_{t}}{\ln 2}  \right|<\frac{\pi}{\ln 2}, 
\end{equation}
for all $s\in \mathbb S$ and $t<\delta_{1}$. Denote $\varepsilon_{t}(s)=s-\Phi_{t} (s)-\frac{2\pi i }{\ln 2} k_{t}$. Since $2^{\varepsilon_{t}(s)}$ converges to $1$ uniformly in $s\in \mathbb S$ as $t$ goes to $0$, by \eqref{Eq:gordo1}, we have that  $\varepsilon_{t}(s)$ converges to $0$ uniformly in $s\in \mathbb S$ as $t$ goes to $0$. In particular, there is  $\delta_{2}<\delta_{1}$ such that  
\begin{equation}\label{Eq:gordo2}
|\varepsilon_{t}(s)|<1/2, \quad \textrm{ for all } t<\delta_{2} \ \textrm{ and } s\in \mathbb S.
\end{equation}

Our goal is to show that whenever $t$ is small enough, it holds $k_t=0$. This would prove that $\Phi_t(s)\to s$ uniformly on the vertical strip $\mathbb S(\varepsilon,\eta)$. To establish this, firstly  we have to prove two statements. The first one claims that $k_{t+\tau}=k_t+k_{\tau}$ whenever $t,\tau\in[0,\delta_{3})$ for some $ 0<\delta_{3}<\delta_{2}/2$. The second one states that, for every $k\in\Z$, the sets
\[
F_k=
\{t\in [0,\delta_{3}): k_t=k\}
\]
are closed. Let us see first how the conclusion follows from these two claims. The family $\{F_k\}_{k\in \Z}$ is a countable covering of the interval $[0,\delta_{3})$ which, clearly, has non-empty interior. If each $F_k$ is closed, then, Baire's category theorem guarantees the existence of a $\tilde k$ such that $F_{\tilde k}$ has non-empty interior. Thus, there exist $t\in F_{\tilde k}$ and  $\delta_{4}>0$ such that $[t,t+\delta_{4}]\subset F_{\tilde k}$. Therefore, thanks to the first claim, 
\[
\tilde k=k_{t+\delta'}=k_{t}+k_{\delta'}, \quad\text{for every $\delta'\in[0,\delta_{4}]$}.
\]
Since $t\in F_{\tilde k }$, we have that $k_{\delta'}=0$ for every $\delta'\in[0,\delta_{4}]$. That is,  $k_t=0$ for every $t\in[0,\delta_{4})$, as desired. This  shows that $a)$ implies $b)$ since given a compact set  in $\C_{+}$ there are $0<\varepsilon<\eta$ such that  the strip $\mathbb S(\varepsilon,\eta)$ contains the compact set.

  Let us prove the first claim. Using the semigroup structure we find that
\begin{equation}\label{Eq:gordo3}
\begin{split}
   0&= \Phi_t(\Phi_{\tau}(s))-\Phi_{t+\tau}(s)\\
   &=\Phi_{\tau}(s)-\frac{2k_{t}\pi i}{\ln 2}-\varepsilon_t(\Phi_{\tau}(s))-(s-\frac{2k_{t+\tau}\pi i}{\ln 2}-\varepsilon_{t+\tau}(s))\\
   &=\frac{2\pi i}{\ln 2}(k_{t+\tau}-k_{t}-k_{\tau})-\varepsilon_{\tau}(s)+\varepsilon_{t+\tau}(s)-\varepsilon_t(\Phi_{\tau}(s)).
   \end{split}
\end{equation} 
Take $s_{0}=(\varepsilon+\eta)/2$ and $t<\delta _{1}$, by \eqref{Eq:gordo1}, 
$$
\varepsilon< s_{0}\leq \Re \Phi_{t}(s_{0})\leq s_{0}+\frac{\pi}{\ln 2}=\frac{\varepsilon+\eta}{2}+\frac{\pi}{\ln 2}<\eta ,
$$
where we have used that $\varepsilon+\frac{2\pi}{\ln 2}<\eta $.
 Therefore, for $t$ and $\tau$ small enough such that $t+\tau<\delta_{2}$, by \eqref{Eq:gordo2}, we have 
  $$\max\{ |\varepsilon_{\tau}(s_{0})|,|\varepsilon_{t+\tau}(s_{0})|,|\varepsilon_t(\Phi_{\tau}(s_{0}))|\}<\frac{1}{2}.$$
Since $k_{t}$, $k_{\tau}$, and $k_{t+\tau}$ are integers, equality \eqref{Eq:gordo3} shows that $k_{t}+k_{\tau}-k_{t+\tau}=0$.

We move onto the second claim. Let $\{t_n\}$ be a sequence in $F_k$ such that $t_n\to b$ as $n\to\infty$. Without loss of generality, we may assume that $t_n\to b^{-}$. Since $T_{t}$ is strongly continuous on $\mathcal H^{2}$, we have that $m^{-\Phi_{t_{n}}}$ converges to $m^{-\Phi_{b}}$ on $\mathcal H^{2}$ for every $m\in \N$ (see \cite[Page 3, Proposition 1.3]{engels}).  Arguing as the beginning of this proof, we have 
\begin{align}\label{conclusion}
 m^{\Phi_b(s)-\Phi_{t_n}(s)}\to1,\quad \text{as  } n\to\infty
\end{align}
uniformly on $\mathbb S$. 
 For every $n\geq1$,
$$
\Phi_b(s)-\Phi_{t_n}(s)=\varepsilon_{t_{n}}(s)-\varepsilon_{b}(s)+\frac{2\pi i}{\ln 2} (k-k_{b}).
$$
Up to taking a subsequence (that we still denote $\{t_{n}\}$), by  Montel's Theorem,  we may assume that the sequence $\{\varepsilon_{t_{n}}-\varepsilon_{b}\}$ converges uniformly on compact sets of $\mathbb S$ to a function $g$. Moreover,
$$
2^{g(s)+\frac{2\pi i}{\ln 2} (k-k_{b})}=1.
$$
Thus there is $\alpha$ such that $g(s)=\alpha$ for all $s$. Using again  \eqref{Eq:gordo2}, we have that $|\alpha|\leq 1$. Moreover, by \eqref{conclusion}, for each $m$ there is $q_{m}\in \Z$ such that 
$$
\alpha+\frac{2\pi i}{\ln 2} (k-k_{b})= \frac{2\pi i}{\ln m} q_{m}.
$$
This implies that $\frac{q_{m}}{\ln m}$ is constant. But this can happen only if $q_{m}=0$ for all $m$ (otherwise, we would get that $\frac{\ln 2}{\ln 3}$ is a rational number). Thus, $\alpha+\frac{2\pi i}{\ln 2} (k-k_{b})=0$. Since $|\alpha|\leq 1$, we get that $k=k_{b}$ and $b\in F_{k}$.

With this we finish the proof of $a)$ implies $b)$. Nevertheless we notice that we have proved that $a)$ implies something stronger than $b)$. Namely, that $\Phi_{t}$ converges to the identity map, as $t$ goes to $0$,  uniformly in vertical strips of $\C_{\varepsilon}$ for every $\varepsilon>0$. We will use this fact later on in the proof of $a)$ implies $c)$.

We now prove $b)$ implies $a)$. 
Using \cite[Theorem 1.6]{engels}, we know that the strong continuity of the semigroup $\{T_t\}_{t\geq0}$ is equivalent to the continuity in the weak operator topology. Therefore, this is what we are going to prove. Let $f\in\mathcal{H}^2$. By Gordon and Hedenmalm's Theorem \ref{gordon}, since $\Phi_t\in\mathcal{G}$ for every $t\geq0$, we know that each $\Phi_t$ defines a bounded composition operator in the Hilbert space $\mathcal{H}^2$. This fact together with Remark \ref{norm}  yield
\[
\|C_{\Phi_t}f\|_{\mathcal{H}^2}=\|f\circ\varphi_t\|_{\mathcal{H}^2}\leq\|C_{\varphi}\|\|f\|_{\mathcal{H}^2}\leq\|f\|_{\mathcal{H}^2}.
\]
Therefore, the set $\{f\circ\Phi_t\}_{t\geq0}$ is bounded in $\mathcal{H}^2$. This guarantees the existence of a subsequence $\{t_k\}_{k\in\N}$, $t_k\to0$ as $k\to\infty$ and $g\in\mathcal{H}^2$, such that
\[
f\circ\Phi_{t_k}\xrightarrow{\enskip w\enskip} g.
\]
Now, given that the pointwise evaluation functional $\delta_s$ is bounded for $s$ in $\C_{1/2}$ (see Lemma \ref{pointwiseevaluaction}), the weak convergence implies that
 \[
 f\circ\Phi_{t_k}(s)\to g(s),\quad k\to\infty,
 \]
 for every $s\in \C_{1/2}$. On the other hand,  $f\circ\Phi_{t}(s)\to f(s)$, as $t$ goes to $0$, for every $s\in \C_{1/2}$. This  forces $f(s)=g(s)$ for every $s\in\C_{1/2}$, so $T_tf\xrightarrow{w} f$ as $t\to0$. By \cite[Theorem 1.6]{engels},  the claim is proven.

By the definition of continuous semigroup, $c)$ implies $b)$ is obvious so that it remains to show that $a)$ implies $c)$. We already  know that $a)$ is equivalent to $b)$. Then, by Proposition \ref{cste1}, we have that $c_t=1$. Moreover,  this implies that $\Phi_t(s)\to s$ as $t\to0^+$, uniformly in vertical strips of $\C_{\varepsilon}$. This already implies $c)$. Indeed, since $\Phi_t\in\mathcal{G}$ and $c_t=1$ for every $t>0$, $\Phi_t(s)=s+\varphi_t(s)$. Therefore, $\varphi_t\to0$ as $t\to0^+$ uniformly on vertical strips of $\C_{\varepsilon}$ for every $\varepsilon>0$.
The Dirichlet series  $\varphi_{t}$ sends $\C_{+}$ into its closure. Thus by Theorem \ref{queffelecs}, it is bounded on  $\C_{\varepsilon}$ for every $\varepsilon>0$. Therefore, its supremum on $\C_{\varepsilon}$ coincides with the supremum in the vertical strip $\varepsilon<\text{Re}(s)<\sigma_0$ for any $\sigma_{0}>\varepsilon$. Thus, $\|\varphi_{t}\|_{\mathcal H^{\infty}(\C_{\varepsilon})}$ tends to zero as $t$ goes to $0$, and we are done.
\end{proof}

\begin{remark}
Let us see how the semigroup structure is essential for $a)$ implies $b)$. Take a sequence $\{x_n\}_n$ of real numbers such that $|x_n|\to\infty$ when $n\to\infty$ and  $m^{-ix_n}\to1$ as $n\to\infty$ for all natural number $m$. 
The existence of such sequence $\{x_{n}\}$ is guaranteed by Kronecker's Lemma. Indeed, if $\{p_{j}\}$ is the sequence of prime numbers, it is enough to see that $p_{j}^{-ix_n}\to1$ as $n\to\infty$ for all $j$. By Kronecker's Lemma, see \cite[Proposition 3.4]{peris}, for each $n\in \N$, the set 
$$
\{(p_{1}^{-ix},p_{2}^{-ix}, \dots, p_{n}^{-ix}): \  x\in \R\}
$$ is dense in $\T^{n}$. Since the set $
\{(p_{1}^{-ix},p_{2}^{-ix}, \dots, p_{n}^{-ix}): \  x\in [-n,n]\}
$ is compact with empty interior (whenever $n\geq 2$), we have that $$
\{(p_{1}^{-ix},p_{2}^{-ix}, \dots, p_{n}^{-ix}): \  |x|\geq n\}
$$ is also dense in $\T^{n}$. Thus, we can find $x_{n}$ such that $|x_{n}|\geq n$ and
$$
|p_{j}^{-ix_{n}}-1|\leq \frac{1}{n}
$$
for $j=1, \dots, n$.

Now define $T_nf=f\circ \Phi_{n}$, where $\Phi_{n}(s)=s+ix_n$ for $s\in \C_{+}$. Then, $T_nf\to f$ in $\mathcal{H}^2$ as $n\to\infty$ because if $f(s)=\sum_{m=1}^{\infty}a_{m}m^{-s}\in \mathcal H^{2}$,  using the Dominated Convergence Theorem,  we obtain
$$
\lim_{n\to \infty}\|f-f\circ \Phi_{n}\|_{\mathcal H^2}^{2}=\lim_{n\to\infty}\sum_{m=1}^{\infty} |a_{m}|^{2}|1-m^{-ix_{n}}|^{2}=0.
$$
However, by the definition of $\{x_n\}_n$, $\Phi_n$ does not converges to the identity map  as $n\to\infty$. 

\end{remark}
Theorem \ref{gordo1} still holds for $1\leq p<\infty$.

\begin{theorem}\label{gordo11}
Let $1\leq p<\infty$. Let $\{\Phi_t\}_{t\geq0}$ be a semigroup of analytic functions, such that $\Phi_t\in\mathcal{G}$ for every $t>0$ and denote by $T_{t}$ the composition operator $T_t(f)=f\circ\Phi_t$. Then, the following assertions are equivalent:
\begin{enumerate}[a)]
    \item $\{T_t\}_{t\geq0}$ is a strongly continuous semigroup in $\mathcal{H}^p$.
    \item $\{\Phi_t\}_{t\geq0}$ is a continuous semigroup.
    \item $\Phi_t(s)\to s$, as $t$ goes to $0$,  uniformly in $\C_{\varepsilon}$, for every $\varepsilon>0$.
\end{enumerate}
\end{theorem}

\begin{proof} Bearing in mind Theorem \ref{gordo1}, we only have to prove the equivalence between $a)$ and $b)$.

Consider a continuous semigroup of analytic functions $\{\Phi_t\}_{t>0}$ in the half-plane $\C_+$, whose elements are in the class $\mathcal{G}$. We know that, necessarily, $c_t=1$ for every $t$. Now, Bayart's Theorem \ref{hp} guarantees the boundedness of the composition operator $T_{t}$ for every $t\geq 0$ in $\mathcal{H}^p$. Not only that, but the operator is a contraction (see Remark \ref{norm}). Bearing in mind these facts, the proof of this implication can be adapted from the one of Theorem \ref{gordo1} except for the case $p=1$, since $\mathcal{H}^1$ is not reflexive. To solve this little inconvenience, we use the density of $\mathcal{H}^2$ in $\mathcal{H}^1$ with respect to the $\mathcal{H}^1$ norm. Indeed, by the $\mathcal{H}^p$ spaces inclusion, we have that 
\[
\|f\circ{\Phi_t}-f\|_{\mathcal{H}^1}\to0,\quad \text{as  } t\to0^+
\]
for every $f\in\mathcal{H}^2$. Now, given $f\in\mathcal{H}^1$ and $\varepsilon>0$, there is $g\in\mathcal{H}^2$ such that
$\|f-g\|_{\mathcal{H}^1}<\varepsilon/2$. Therefore, using Remark \ref{norm}, we have
$$
\|T_{t}(f)-f\|_{\mathcal{H}^1}\leq \|T_{t}(f)-T_{t}(g)\|_{\mathcal{H}^1}+\|T_{t}(g)-g\|_{\mathcal{H}^1}+\|g-f\|_{\mathcal{H}^1}\leq \varepsilon+\|T_{t}(g)-g\|_{\mathcal{H}^1}.
$$
Thus, $\limsup_{t\to 0}\|T_{t}(f)-f\|_{\mathcal{H}^1}\leq\varepsilon$. The arbitrariness of $\varepsilon$ shows that $\lim_{t\to 0}\|T_{t}(f)-f\|_{\mathcal{H}^1}=0$.

Reciprocally, given a semigroup of elements in the class $\mathcal{G}$, if we have a strongly continuous semigroup of composition operators from $\mathcal{H}^p$ into $\mathcal{H}^p$,  the proof of $a)$ implies $b)$ in Theorem \ref{gordo1} still holds for this range of $p$. Indeed, the proof only requires to consider the Dirichlet series $m^{-s}$, with $m\in \N$, which belongs to any $\mathcal{H}^p$ and the strong continuity of the operator semigroup in some $\mathcal{H}^p$. 
\end{proof}
\begin{remark}
For the case $p=\infty$, the previous theorem is no longer true. The problem for the equivalence lies on the implication ``$b)\Rightarrow a)$''. As we shall see, the only strongly continuous semigroup of composition operators $\{T_t\}_{t\geq0}$ is the trivial one (see Theorem \ref{gordoinfty}). 
\end{remark}

\section{The infinitesimal generator} \label{sec:infgenerator}   

In the study of both semigroups of operators and of holomorphic functions, the infinitesimal generators play a fundamental role. See, i.e.,\cite[Chapter 10]{manoloal} for the case of holomorphic semigroups. Regarding semigroups of operators we refer the reader either to \cite[Chapter II]{engels} or \cite[Chapter 13]{rudin}. The aim of this section is to characterize the infinitesimal generators of continuous semigroups in the Gordon-Hedenmalm class. As a byproduct, we will describe the infinitesimal generator of a strongly continuous semigroup of composition operators in Hardy spaces of Dirichlet series.  

Let us recall that given  a Banach space $X$ and $\{T_t\}_{t\geq0}$  an operator semigroup where $T_t:X\to X$, the infinitesimal generator of the semigroup is defined as
\begin{equation}\label{limiton}
    Af=\lim_{t\to0^+}\frac{T_tf-f}{t},
\end{equation}
where the convergence is considered in the norm topology. We denote by $D(A)$ the set of all $f\in X$ such that the limit \eqref{limiton} exists. 

A classical result from general semigroup theory guarantees that if the semigroup of composition operators $\{T_t\}_{t\geq0}$ is strongly continuous, then $D(A)$ is dense in the space $X$ (see, i.e. \cite[Page 37, Theorem 1.4]{engels}). 

\subsection{Infinitesimal generator of a semigroup of holomorphic functions in $\mathcal G$}
A remarkable result of Berkson and Porta \cite{porta} asserts that
each continuous semigroup of holomorphic self-maps of $\C_{+}$ is
locally uniformly differentiable with respect to the parameter
$t\geq 0.$  That is, there exists
\begin{equation}\label{stress}
H(s)=\lim_{t\to0^+}\frac{\Phi_t(s)-s}{t}, \qquad \textrm{ for all } s\in \C_{+}
 \end{equation}
and such limit is uniform on compact sets of $\C_{+}$. In particular, $H$ is holomorphic.
Moreover, $\Phi _{t}$ is the solution of the Cauchy problem:
\begin{equation}\label{cauchy}
\frac{\partial \Phi _{t}(s)}{\partial t}=H(\Phi _{t}(s))\quad
\mbox{and} \quad \Phi _{0}(s)= s\in \C_{+}.
\end{equation}
The function $H$ is called the  {\sl infinitesimal generator} of
the semigroup $\left\{ \Phi _{t}\right\} .$ 
In fact, in \cite[Theorem 2.6]{porta}, it is proved that $H$ is the infinitesimal generator of a continuous semigroup of analytic functions with Denjoy-Wolff point $\infty$ if and only if $H(\C_{+})\subset\overline{ \C_{+}}$. 

Let us recall that  the Denjoy-Wolff point of  a semigroup in the Gordon-Hedenmalm class is $\infty$. Thus, its infinitesimal generator is a holomorphic function sending the right half-plane into its closure. Clearly the converse of this assertion does not hold.  The main result  of this section is the following  characterization of the infinitesimal generators of continuous semigroups in the Gordon-Hedenmalm class.

\begin{theorem}\label{gcoro}
Let $H:\C_+\to\overline{\C}_+$ be analytic. Then, the following statements are equivalent:
\begin{enumerate}[a)]
    \item $H$ is the infinitesimal generator of a continuous semigroup of elements in the class $\mathcal{G}$.
    \item $H\in\mathcal{H}^{\infty}(\C_{\varepsilon}),$ for all $\varepsilon>0$.
    \item $H\in\mathcal{D}$.
\end{enumerate}
\end{theorem}

The proof of this result will be given at the end of this subsection (see page \pageref{proofgcoro}). In fact, it will be an easy consequence of some more general results.

\begin{theorem}\label{gordo2}
Let $\{\Phi_t\}_{t\geq0}$ be a continuous semigroup of analytic functions in $\C_+$ such that $\Phi_t\in\mathcal{G}$ for every $t>0$ and $H$ its infinitesimal generator. Then, $H\in\mathcal{H}^{\infty}(\C_{1/2+\sigma})$, for every $\sigma>0$.
In addition, if  $\Phi_t(s)=s+\varphi_{t}(s)$, $\Lambda$ is a multiplicative semigropup of $\N$,   and $\varphi_t\in\mathcal{D}_{\Lambda}$ for all $t$, then  $H\in\mathcal{H}^{\infty}_{\Lambda}(\C_{1/2+\sigma})$.
\end{theorem}
For clarity, we have extracted the following lemma for the proof of Theorem \ref{gordo2}. The proof of this lemma is similar to the one of \cite[Lemma 10.29]{rudin_real_1987}.

\begin{lemma}\label{uniforme}
Let $f\in \mathcal H^{2}$ and  $1/2<\alpha<\beta$. Consider the vertical strip $\Omega=\{s\in\C: \alpha<\Re(s)<\beta\}$ and define $F:\Omega\times\Omega\to\C$ by
\begin{equation}\label{funcionf}
F(z,w)=
\begin{cases}
\frac{f(z)-f(w)}{z-w}&\quad\text{if $z\not=w$},\\
\qquad\\
f'(z)&\quad \text{if $z=w$}.
\end{cases}
 \end{equation}
Then, the function $F$ is uniformly continuous on $\Omega\times\Omega$.
\end{lemma}
\begin{proof}
Choose $\varepsilon>0$. By Lemma \ref{pointwiseevaluaction}, there is $M>0$ such that $|f''(s)|\leq M$ for all $s\in \Omega$. Take $\delta=\varepsilon/M$ and  $z_0,z_1\in\Omega$ such that $|z_0-z_1|<\delta$. Consider the curve 
\[
\gamma(t)=tz_0+(1-t)z_1,\quad t\in[0,1].
\]
Then, 
\begin{align*}
|f'(z_0)-f'(z_1)|=\left|\int_0^1f''(\gamma(t))(z_0-z_1)dt\right| \leq M|z_0-z_1|<\varepsilon. 
\end{align*}
Now,  take $(z_0,w_0),(z_1,w_1)\in \Omega\times\Omega$ such that $|z_0-z_1|<\delta$ and $|w_0-w_1|<\delta$. Notice that
\begin{align*}
    |F(z_0,w_0)-F(z_1,w_1)|&=\left|
    \int_0^1\left(f'(w_0+t(z_0-w_0))-f'(w_1+t(z_1-w_1))\right)dt
    \right|\\
    &\leq\int_0^1|f'(w_0+t(z_0-w_0))-f'(w_1+t(z_1-w_1))|dt<\varepsilon,
\end{align*}
since $|w_0+t(z_0-w_0)-w_1-t(z_1-w_1)|\leq\delta$ and we are done.
\end{proof}

\begin{proof}[Proof of Theorem \ref{gordo2}]
Fix $\sigma >0$. By Lemma \ref{pointwiseevaluaction}, there is a constant $C(\sigma)>0$ such that
$|f(s)|\leq C(\sigma)\|f\|_{\mathcal H^{2}}$ and $|f'(s)|\leq C(\sigma)\|f\|_{\mathcal H^{2}}$ for all $f\in \mathcal H^{2}$
and  $s\in \C_{\sigma+1/2}$. 

Given $\varepsilon>0$ and $h(s)=2^{-s}$, take $f\in \mathcal H^{2}$ a function in the domain of the infinitesimal operator of the strongly continuous semigroup $T_{t}=C_{\Phi_{t}}$ in $\mathcal H^{2}$ such that $\|f-h\|_{\mathcal{H}^2}<\varepsilon$. Moreover, we have
\begin{equation}\label{epsilon}
    |f(s)-h(s)|\leq C(\sigma)\varepsilon
\end{equation}
and 
\begin{equation}\label{epsilon1}
    |f'(s)-h'(s)|\leq C(\sigma)\varepsilon,
\end{equation}
for all $s\in \C_{1/2+\sigma}$. Thus, using \eqref{epsilon} and \eqref{epsilon1} for $\sigma/2$ and $\varepsilon$ small enough, 
both $|f|$ and $|f'|$ are bounded below by a positive constant in the vertical strip $\Omega=\{s\in\C: \frac12+\frac{\sigma}{2}<\Re(s)<\frac12+3\sigma \}\subset \C_{1/2+\sigma/2}$.

Take the function $F$ introduced in Lemma \ref{uniforme} associated with the function $f$ and consider the function
\[
K(z,w)=
\begin{cases}
\frac{2^{-z}-2^{-w}}{z-w,}&\quad\text{if $z\not=w$},\\
\qquad\\
-\ln(2)2^{-z},&\quad \text{if $z=w$}.
\end{cases}
\]
Then, given $z,w\in\Omega$,
\begin{equation}\label{Eq-luis}
\begin{split}
    |F(z,w)-K(z,w)|
    &\leq \int_0^1|f'(w+t(z-w))-h'(w+t(z-w))|dt\\&
    \leq C(\sigma/2)\|f-h\|_{\mathcal{H}^2}\leq C(\sigma/2)\varepsilon.
    \end{split}
\end{equation}
Let us give a lower bound of $|K|$,
\begin{align*}
   |K(z,w)|= \left|
    \frac{2^{-z}-2^{-w}}{z-w} \right|
    =
    2^{-\text{Re}(z)}
    \frac{|1-2^{z-w}|}{|z-w|}
   \end{align*}
Note that $|1-2^u|/|u|$ tends to $\ln2$ as $|u|\to0$. Thus, there is $\delta_{0}>0$ such that  
\begin{equation*}\label{wifi}
|K(z,w)|\geq 2^{-\text{Re}(z)}\frac{\ln2}{2},
 \end{equation*}
 whenever  $|z-w|\leq\delta_0$. Therefore, there exists $C>0$ such that $|K(z,w)|>C$ for all $z,w\in \Omega$ with $|z-w|\leq\delta_0$.  Taking $\varepsilon\leq \frac{C}{2C(\sigma/2)}$, \eqref{Eq-luis} necessarily forces $|F(z,w)|>C/2$ for any $z,w\in\Omega$ such that $|z-w|\leq\delta_0$.
 
By Theorem \ref{gordo1},   $\Phi_{t}$ converges to the identity map uniformly in $\Omega$. Hence, for $t$ small enough and $1/2+\sigma/2<\Re s<\frac12+2\sigma$, it holds $ \Phi_{t}(s)\in \Omega$ and $| \Phi_{t}(s)-s|<\delta_{0}$. Therefore, $|F(\Phi_{t}(s),s)|\geq C/2$ and $\lim_{t\to 0}F(\Phi_{t}(s),s)=F(s,s)=f'(s)$ uniformly in $\Omega'=\{s\in\C: \frac12+\sigma<\Re(s)<\frac12+2\sigma\}$. Thus $\lim_{t\to 0}\frac{1}{F(\Phi_{t}(s),s)}=\frac{1}{f'(s)}$ uniformly in $\Omega'$.
 
By the very definition of infinitesimal generator, there exists $g\in \mathcal H^{2}$ such that
\begin{equation}\label{keypoint}
\lim_{t\to 0} \frac{f(\Phi_t(s))-f(s)}{t}=g(s)
\end{equation}
in the norm of $\mathcal H^{2}$. Then 
\[
 \lim_{t\to 0}\frac{f(\Phi_t(s))-f(s)}{t}= g(s),
\]
uniformly on the vertical strip $\Omega'$.

 The introduction of the function $F$ allows us to rewrite the incremental quotient in \eqref{stress} as
\begin{align*}
  g_{t}(s):= \frac{\varphi_t(s)}{t}=\frac{\Phi_t(s)-s}{t}
    =
    \frac{f(\Phi_t(s))-f(s)}{t}\cdot \frac{1}{F(\Phi_t(s),s)}.
\end{align*}
Putting all together, and using that both $g$ and $\frac{1}{f'}$ are bounded on $\Omega '$, we have 
\[
H(s)=\lim_{t\to0^+} \frac{\Phi_t(s)-s}{t}=\frac{g(s)}{f'(s)},
\]
uniformly on the vertical strip $\Omega '$. 
$H$ is a holomorphic function in $\C_{+}$ and $g_{t}$ converges uniformly on $\Omega'$, and then in $\C_{1/2+\sigma}$, to $H$. Thus  $H\in\mathcal{H}^{\infty}(\C_{1/2+\sigma})$ for every $\sigma>0$.

In the case $\varphi_{t}\in \mathcal D_{\Lambda}$ for all $t$, we obtain that $g_{t}\in\mathcal{H}_{\Lambda}^{\infty}(\C_{1/2+\sigma}),$ for all $t$, and thus $H\in\mathcal{H}_{\Lambda}^{\infty}(\C_{1/2+\sigma})$.
\end{proof}

The converse of  Theorem \ref{gordo2}  requires the next technical theorem.   Following the ideas in the Picard-Lindel\"of  Theorem, 
 for $\sigma>0$ and $\delta>0$ small, we consider the space $X$ consisting on the collection of functions $f:\C_{1+\sigma}\times[0,\delta]\to\overline {\C_{1+\sigma}}$ satisfying the following three properties
\begin{enumerate}[i)]
    \item $f$ is continuous on $\C_{1+\sigma}\times[0,\delta]$;
    \item $s \mapsto f(s,t)-s\in\mathcal{H}^{\infty}(\C_{1+\sigma})$ for each $ t\in [0,\delta]$;
    \item The map given by
  \begin{align*}
      [0,&\delta]\to \mathcal{H}^{\infty}(\C_{1+\sigma})\\
      &t\mapsto f(s,t)-s
        \end{align*}
        is continuous.
          \end{enumerate}
Notice that $X$ depends on $\sigma $ and $\delta$ but we do not write explicitly such dependence in order to simplify the exposition. We endow $X$ with the distance defined for $f,g\in X$ as
\begin{equation}\label{metric}
    d(f,g)=\|f-g\|_{\infty}=\sup_{\substack{s\in \C_{1+\sigma}, \\ t\in[0,\delta]}}|f(s,t)-g(s,t)|.
\end{equation}
Note that conditions ii) and iii) guarantee that $d:X\times X\to[0,\infty)$. In fact, $(X,d)$ is complete. Indeed, let $\{f_n\}_n$ be a Cauchy sequence of elements in $X$. 
Since 
\begin{equation*}
   |f_n(s,t)-f_m(s,t)|\leq d(f_n,f_m)
\end{equation*}
for all $s$ and $t$, we have that the sequence $\{f_n(z,t)\}_n$ is Cauchy in $\C$. This guarantees the existence of the limit 
\begin{equation*}
    f(s,t):=\lim_{n\to\infty}f_n(s,t).
\end{equation*}
A standard argument shows that this convergence is uniform in $\C_{1+\sigma}\times[0,\delta]$, so that $f$ is continuous. 
Regarding the second property, the uniform limit of bounded Dirichlet series in $\C_{1+\sigma}$ yields again a bounded Dirichlet series in the same half-plane. Eventually, $t\mapsto f(z,t)-z$ maps the interval $[0,\delta]$ into the algebra $\mathcal{H}^{\infty}(\C_{1+\sigma})$ and, being $f$ the uniform limit of continuous $f$, the map is continuous too. Therefore, the metric space $(X,d)$ is complete.
\begin{proposition}\label{gordo4}
Let $H:\C_+\to\overline{\C}_+$ analytic and such that $H\in\mathcal{H}^{\infty}(\C_{1/2+\sigma})$ for every $\sigma>0$.
We define the operator $T$ in $X$ given by
\[
Tf(s,t)=s+\int_0^tH(f(s,\tau))d\tau,\qquad f\in X, \quad (s,t)\in \C_{1+\sigma}\times[0,\delta].
\]
Then, there is $\delta=\delta (H, \sigma)$ small enough, such that
\begin{enumerate}
    \item $T:X\to X$,
    \item $T$ is contractive.
\end{enumerate}
\end{proposition}

\begin{proof}
Observe that for $s\in\C_{1/2+2\sigma}$, the uniform convergence of the Dirichlet series defining $H$ allows us to write
\begin{align*}
    -H'(s)=\sum_{n=1}^{\infty}a_n\log n\, n^{-s}.
\end{align*}
Now, since $H\in\mathcal{H}^{\infty}(\C_{1/2+\sigma})$, by the Cauchy integral formula, we obtain that $H'$ is bounded in $\C_{1+\sigma}$ for every $\sigma>0$. 
Fix $\sigma>0$. Take $$M=\max \{\sup_{s\in \C_{1+\sigma}}|H'(s)|, \sup_{s\in \C_{1+\sigma}}|H(s)|\}$$ and $\delta<1/M$.

Let us see first that $Tf$ maps $\C_{1+\sigma}\times[0,\delta]$ into $\C_{1+\sigma}$. This is indeed the case because
\begin{align*}
    \Re(Tf(s,t))=\Re(s)+\int_0^t\Re(H(f(s,\tau)))d\tau> 1+\sigma
\end{align*}
where we have used that  $\Re(H)\geq 0$. For the continuity, notice that 
\begin{align*}
    |Tf(s,t)-Tf(s_0,t_0)|&\leq |s-s_0|+\left|
    \int_0^tH(f(s,\tau))d\tau-\int_0^{t_0}H(f(s_0,\tau))d\tau
    \right|\\
    &\leq  |s-s_0| +\int_0^t\left|
    H(f(s,\tau))-H(f(s_0,\tau))
    \right|d\tau+\int_t^{t_0}|H(f(s_0,\tau))|d\tau\\
    &\leq  |s-s_0|+ M|t-t_{0}|+M \int_0^t\left|
    f(s,\tau)-f(s_0,\tau)
    \right|d\tau.
\end{align*}
From these inequalities and the very definition of $X$, we deduce that $Tf$ is continuous in  $\C_{1+\sigma}\times[0,\delta]$.

Now, we verify that $s\mapsto Tf(s,t)-s=\int_0^tH(f(s,\tau))d\tau$ belongs to $\mathcal{H}^{\infty}(\C_{1+\sigma})$. In virtue of Theorem \ref{menudeo} we deduce that $H\circ f(\cdot, \tau)\in\mathcal{D}$ for every $\tau$ and since $H\in\mathcal{H}^{\infty}(\C_{1/2+\sigma})$,  we have that $F(\cdot, \tau)=H\circ f(\cdot, \tau) \in\mathcal{H}^{\infty}(\C_{1+\sigma})$. Moreover, using again that $H'$ is bounded, we deduce that the function $\tau\mapsto   F(\cdot, \tau)$ is continuous. Thus, using that $\mathcal{H}^{\infty}(\C_{1+\sigma})$ is a Banach space, we have that  $\int_0^tH(f(\cdot,\tau))d\tau$ also belongs to $\mathcal{H}^{\infty}(\C_{1+\sigma})$ and that the map $t\mapsto \int_0^tH(f(\cdot,\tau))d\tau$ is continuous.
Thus, we have obtained (1).

For the contractivity, let $f_1,f_2\in X$. Then,
\begin{align*}
 |Tf_1(s,t)-Tf_2(s,t)| &\leq   \int_0^t\left|
    H(f_1(s,\tau))-H(f_2(s,\tau))
    \right|d\tau
    \leq M\int_0^t |f_1(s,\tau)-f_2(s,\tau)|d\tau\\&\leq t M\|f_1-f_2\|_{\infty}=tMd(f_1,f_2)\leq \delta M d(f_1,f_2).
\end{align*}
Since $\delta M<1$, we get the contractivity.
\end{proof}

\begin{theorem}\label{gordo3}
Let $H:\C_+\to\overline{\C}_+$ analytic and such that $H\in\mathcal{H}^{\infty}(\C_{1/2+\sigma})$ for every $\sigma>0$. Then, $H$ is the infinitesimal generator of a continuous semigroup $\{\Phi_t\}_{t\geq0}$ where
\begin{equation*}
    \Phi_t(s)=s+\varphi_{t}(s)
\end{equation*}
and $\varphi_{t}\in\mathcal{D}$, for all $t$, that is, $\{\Phi_{t}\}$ is a continuous semigroup in the Gordon-Hedelmann class $\mathcal G$.
\end{theorem}
\begin{proof} 
We notice that whenever the boundary of $\C_+$ is attained, then $H$ is constant and the result is straightforward. Thus, we assume that $H:\C_+\to \C_{+}$. By the Berkson-Porta Theorem \cite[Theorem 2.6]{porta}, there exists a unique continuous semigroup $\{\Phi_t\}_{t\geq0}$ in $\C_{+}$ such that $H$ is its infinitesimal generator. In particular, the map $[0,\delta)\ni t\mapsto \Phi_{t}(s)$ is the unique solution of the Cauchy problem 
\begin{equation*}
\frac{\partial \Phi _{t}(s)}{\partial t}=H(\Phi _{t}(s))\quad
\mbox{and} \quad \Phi _{0}(s)= s\in \C_{+}.
\end{equation*}

On the other hand, and following the notation introduced in Proposition \ref{gordo4}, by the Banach Fixed Point Theorem, there are $\delta>0$, a continuous function $f:\C_{2}\times[0,\delta]\to\overline {\C_{2}}$ satisfying the following three properties
\begin{enumerate}[i)]
 \item $s \mapsto f(s,t)-s\in\mathcal{H}^{\infty}(\C_{2})$ for each $ t\in [0,\delta]$;
    \item the map given by $ t\mapsto  f(s,t)-s$ from $[0,\delta] $ to  $ \mathcal{H}^{\infty}(\C_{2})$ is continuous;
 \item and $f$ is a fixed point of the operator introduced in  Proposition \ref{gordo4}. That is 
\[
f(s,t)=s+\int_0^tH(f(s,\tau))d\tau,\qquad f\in X, \quad (s,t)\in \C_{2}\times[0,\delta].
\]
  \end{enumerate}
 Write $g_{t}(s)=f(s,t)$ for $ s\in \C_{2}$ and $t\in [0,\delta]$. Clearly,  the map $[0,\delta)\ni t\mapsto g_{t}(s)$ is a solution of the Cauchy problem 
\begin{equation*}
\frac{\partial g _{t}(s)}{\partial t}=H(g_{t}(s))\quad
\mbox{and} \quad g _{0}(s)= s\in \C_{2}.
\end{equation*}
Thus, by the uniqueness of the Cauchy problem, $\Phi_{t}(s)=g_{t}(s)$ for $0\leq t<\delta$ and $\Re s>2$. In particular, this implies that $s\mapsto \Phi_{t}(s)-s$ is a Dirichlet series for $t<\delta$. That is, for those values of $t$, $\Phi_{t}\in \mathcal G$. Finally, by Proposition \ref{colorario} and the very definition of semigroup, we deduce that $\Phi_{t}\in \mathcal G$ for all $t$. 
\end{proof}

\begin{remark}
If the function $H$ in  Theorem \ref{gordo3} belongs to  $\mathcal{H}^{\infty}_{\Lambda}(\C_{1/2+\sigma})$, with $\Lambda $ a multiplicative semigroup of natural numbers, our proof can be easily adapted to get that  $\varphi_{t} \in \mathcal D_{\Lambda}$ for all $t>0$.
\end{remark}

One way to provide examples of continuous semigroups is the following. Consider a holomorphic function $G:\D\to \overline{\C_{+}}$, with $G(z)=\sum_{n=0}^{\infty} a_{n}z^{n}$, $z\in \D$. Fix an integer $q\geq 2$ and define 
$$
H(s)=G(q^{-s})=\sum_{n=0}^{\infty} a_{n}(q^{n})^{-s}, \quad s\in \mathbb C_{+}. 
$$
By Theorem \ref{gordo3}, there exists a continuous semigroup $\{\Phi_{t}\}_{t\geq 0}$ in $\mathcal G$ such that $H$ is its infinitesimal generator. Moreover, by the previous remark, as $\Lambda=\{q^{n}: \, n\geq 0\}$ is a multiplicative semigroup, it is easy to deduce that for every $t\geq 0$ there exists $g_{t}:\D\to \C_{+}$ holomorphic such that $\Phi_{t}(s)=s+g_{t}(q^{-s})$, $s\in \C_{+}$. We will see a concrete example of this situation in Example \ref{exampleluis}.

\begin{proof}[Proof of Theorem \ref{gcoro}] 
For $a)$ implies $b)$, we  use first Theorem \ref{gordo2} and then Theorem \ref{queffelecs}. The fact that $b)$ implies $c)$ is trivial. For $c)$ implies $a)$, we use Theorem \ref{queffelecs} and we find that $\sigma_u(H)\leq0$. This allows us to apply Theorem \ref{gordo3} and $a)$ follows. \label{proofgcoro}
\end{proof}

\subsection{Infinitesimal generator of a strongly continuous semigroup of composition operators in $\mathcal H^{p}$}
Let $1\leq p<\infty$ and take a strongly continuous semigroup of composition operators $\{T_t\}_t$ given by $T_tf=f\circ\Phi_t$, where $\{\Phi_t\} _{t}$ is a continuos semigroup in the class $\mathcal G$.  Denote by $A$ the infinitesimal generator of $\{T_{t}\}$. By Theorem \ref{gcoro}, the infinitesimal generator of $\{\Phi_t\} _{t}$  is a Dirichlet series $H$ sending $\C_{+}$ in $\overline{\C_{+}}$.  
 Take $f\in\mathcal{D}(A)$. Then, by the very definition of $A$ and the chain rule
\begin{align*}
    Af(s)=\lim_{t\to0^+}\frac{f\circ\Phi_t(s)-f(s)}{t}=f'(s)\frac{\partial}{\partial t}(\Phi_t(s))\Big|_{t=0}=f'(s)H(s)
\end{align*}
whenever $\Re s>1/2$.
Moreover, we have obtained that  $$
D(A)\subset \{f\in \mathcal H^{p}: Hf'\in \mathcal H^{p}\}.
$$ The other inclusion was proved in \cite[Theorem 2]{BCDMPS} in a much more general context using properties of the resolvent of a semigroup of operators. Thus, we have: 

\begin{proposition} \label{Prop:infgen}
Let $1\leq p<\infty$ and take  a strongly continuous semigroup of composition operators $T_tf=f\circ\Phi_t$ in $\mathcal H^{p}$. Then, there is a Dirichlet series $H: \C_{+}\to \overline{\C_{+}}$ such that the infinitesimal generator is given by the operator 
$
A(f)=H f'
$ and
$$
D(A)=\{f\in \mathcal H^{p}: Hf'\in \mathcal H^{p}\}.
$$
\end{proposition}

From the very beginning of our exposition, we have been working with  strongly continuous semigroups. Another standard and useful notion of semigroups of operators are those which are uniformly continuous. Let us recall this notion.  Consider a Banach space $X$ and  a semigroup of operators $\{T_t\}_t$ in $X$. It is said that $\{T_t\}_t$ is uniformly continuous if and only if $T_{t}$ converges, as $t$ goes to $0$, to the identity map in the norm of the space of bounded operators in $X$.  Clearly, every uniformly continuous semigroup is strongly continuos. 
A classical result states that a strongly continuos semigroup with infinitesimal generator $A$ is uniformly continuous if and only if $A$ is bounded and if and only if $D(A)=X$. In such a case, $T_{t}=e^{tA}$ for all $t>0$. See, i.e.  \cite[Corollary 1.5, Page 39]{engels}. We will show that no non-trivial semigroup of composition operators is uniformly continuous on $\mathcal H^{p}$.

\begin{theorem}\label{Thm:uniformly}
$1\leq p<\infty$. Let $\{T_t\}_{t\geq0}$ be a uniformly continuous semigroup of composition operators in $\mathcal{H}^p$. Then, $T_t=\mathrm{Id}$ for every $t\geq0$.
\end{theorem}
\begin{proof}
By Proposition \ref{Prop:infgen}, if the semigroup is strongly continuous with infinitesimal generator $A$, then 
$$
D(A)=\{f\in \mathcal H^{p}: Hf'\in \mathcal H^{p}\},
$$
where  $H: \C_{+}\to \overline{\C_{+}}$ is a Dirichlet series. 
Assume that $A$ is bounded. Notice that, for each $n\geq 2$, the operator $  M_n: \mathcal{H}^p \to \mathcal{H}^p$ given by $M_{n}(f)=n^{-s}f$ is an isometry.
Using this, we have that
 \begin{align*}
    \|A(n^{-s})\|_{\mathcal{H}^p}= \|H(n^{-s})'\|_{\mathcal{H}^p}=\log n\, \|Hn^{-s}\|_{\mathcal{H}^p}=\log n\, \|H\|_{\mathcal{H}^p}.
 \end{align*}
 Since   $\|n^{-s}\|_{\mathcal{H}^p}=1$ for all $n$, we deduce that the operator $A$ is  bounded in $\mathcal{H}^p$ if, and only if, $H\equiv0$. Clearly, this forces $T_t=Id$ for every $t\geq0$, as desired.
\end{proof}

\section{The Koenigs function of semigroups in $\mathcal G$} \label{sec:Koenigs}

In this section, we provide a characterization of the Koenigs function of a given continuous semigroup in the class $\mathcal G$. In fact, this study provides more new dynamic information about the semigroup. Indeed, we prove that, up to automorphims, the functions of a continuous semigroups in the class $\mathcal G$ are of zero hyperbolic step. For a reference to this topic for general semigroups we refer the reader to \cite[Chapter 9]{manoloal}. 

\begin{theorem}\cite[Theorem 9.3.5]{manoloal}
Let $\{\Phi_t\}_{t\geq0}$ be a non-elliptic continuous semigroup of analytic functions in $\C_+$. Then there exists a  univalent function $h:\C_+\to\C$ such that
\begin{equation}\label{modelo}
    h\circ\Phi_t(s)=h(s)+t.
\end{equation}
The function $h$ is unique up to an additive constant. 
\end{theorem}
The function $h$ is known as the Koenigs function of the semigroup. 
The interest about such function is that its study can provide quite useful information about the semigroup. Indeed, since the  function $h$ is univalent, we can recover the semigroup $\{\Phi_t\}_{t\geq0}$ as
\begin{equation*}
    \Phi_t(s)=h^{-1}(h(s)+t).
\end{equation*}
If we differentiate with respect to $t$ in \eqref{modelo} and evaluate at $t=0$, the chain rule gives in the left-hand side
\begin{equation*}
 \frac{\partial}{\partial t}   (h\circ\Phi_t(s))\big|_{t=0}=h'(\Phi_0(s))H(s)=h'(s)H(s),
\end{equation*}
where $H$ is the infinitesimal generator of the semigroup $\{\Phi_t\}$. The right hand-side gives
\[
\frac{\partial}{\partial t}(h\circ\Phi_t(s))\big|_{t=0}=1.
\]
Therefore,
\begin{equation} \label{derivadaKoenigs}
    h'(s)=\frac{1}{H(s)}.
\end{equation}

Before describing the Koenigs function of a semigroup in the class $\mathcal G$, we study some properties of the coefficients of the infinitesimal generator.  In what follows, given a  continuous semigroup $\{\Phi_t\}_{t\geq0}$ of analytic functions in the class $\mathcal{G}$, we write $\Phi_{t}(s)=s +\varphi_{t}(s)$ and $\varphi_t(s)=\sum_{n\geq1}a_n(t)n^{-s}$. Then its infinitesimal generator $H$ is given by 
\begin{equation}\label{efila}
    H(s)=\lim_{t\to0^+}\frac{\varphi_t(s)}{t}
\end{equation}
uniformly in vertical strips of $\C_{1/2+\sigma}$. Hence,  we can write 
\begin{equation}\label{efila1}
    H(s)=\lim_{t\to0^+}\psi_t(s),
\end{equation}
where $\psi_t(s)=\sum_{n\geq1}c_n(t)n^{-s}$, with $c_n(t)=a_n(t)/t$.
Let $\{b_n\}_{n\geq1}$ be the sequence of coefficients of the infinitesimal generator $H$. 
\begin{lemma}\label{hache}
For every $n\in\N$, $b_n=a'_n(0)$.
\end{lemma}
\begin{proof}
For $n\in\N$ fixed, using the integral formula for the coefficient of a Dirichlet series \cite[Proposition 1.9]{peris}, we get
\begin{align*}
    |c_n(t)-b_n|&=\lim_{T\to+\infty}\frac{1}{2T}\left|
    \int_{-T}^T(\psi_{t}(1+iy)-H(1+iy))n^{1+iy}dy
    \right|\\&
    \leq \lim_{T\to+\infty}\frac{1}{2T}\int_{-T}^T\left|\psi_t(1+iy)-H(1+iy)\right|ndy\\
    &\leq n\sup_{s\in\Omega}|\psi_t(s)-H(s)|
    \end{align*}
    where $\Omega=\{s\in\C:\frac 34<\text{Re}(s)<\frac 32\}$. If we let $t\to0^+$, by Theorem \ref{gordo2}, the supremum vanishes and we find that, for each $n\in\N$,
    \[
    c_n(t)\to b_n \quad \text{as $t\to0^+$}.
    \]
    On the other hand, for each $n\in\N$
    \[
    \lim_{t\to0^+}c_n(t)= \lim_{t\to0^+}\frac{a_n(t)}{t}=a'_n(0).
    \]
    The uniqueness of the limit gives the desired conclusion. 
\end{proof}
\begin{lemma}\label{aes}
For every $n$, 
\begin{equation}\label{formula}
    a_n(t+u)=a_n(u)+\sum_{k=1}^{n}a_k(t)b_n^{k}(u),
\end{equation}
where $b_n^k(u)$ stands for the $k^{th}$ coefficient of $n^{-\Phi_u(s)}$.
\end{lemma}
\begin{proof}
The semigroup structure of the $\Phi_t$ forces the Dirichlet series $\varphi_t$ to satisfy the following relation
\begin{equation}\label{quesi1}
    \varphi_{t+u}(s)=\varphi_t(s+\varphi_u(s))+\varphi_u(s),
\end{equation}
for all $t,u\geq 0$ and $s\in \C_{+}$.
Since the space $\mathcal{D}$ is a linear space, we have that the composition $\varphi_t(s+\varphi_u(s))$ is in $\mathcal{D}$ since it can be written as the difference of two Dirichlet series. This can also be deduced from Theorem \ref{menudeo}. In particular, this implies that in a sufficiently remote half-plane, the series converges absolutely. This observation allows us to reorder the series $\varphi_t(s+\varphi_u(s))$ at will, so that
\begin{equation*}
   \varphi_t(s+\varphi_u(s))=\sum_{n=1}^{\infty}a_n(t)n^{-\Phi_u(s)}= \sum_{n=1}^{\infty}a_n(t)\sum_{k=n}b_k^{n}(u)k^{-s}
    =\sum_{n=1}^{\infty}\left(
    \sum_{k=1}^na_k(t)b_n^{k}(u)
    \right)n^{-s}.
\end{equation*}
Knowing this and using the standard procedure to recover the coefficients of a Dirichlet series in \eqref{quesi1}, we find that for every $n$ and $t,u\geq0$ identity \eqref{formula} holds.
\end{proof}
\begin{proposition}\label{paquito}
Let $\{\Phi_{t}\}$ be a continuous semigroup in $\mathcal G$. With the notation introduced above, $\Re b_{1}\geq 0$ and 
$a_1(t)= b_{1 }t$,  for all $t\geq 0$. In addition, 
if the functions of the semigroup $\{\Phi_t\}_{t\geq0}$ are not automorphims of $\C_{+}$, then $\Re(b_1)>0$.
\end{proposition}
\begin{proof}
Since, $\lim_{\text{Re}(s)\to+\infty}\varphi_t(s)=a_1(t)$ and $\varphi_t:\C_+\to\C_+$, we deduce that the map $t\mapsto a_{1}(t)$ is measurable and $\text{Re}(a_1(t))\geq0$ for every $t\geq0$. Now, using \eqref{formula} for $n=1$, we find that
\[
a_1(u+t)=a_1(u)+a_1(t).
\]
That is, $a_1(\cdot)$ is additive and measurable. Then, $a_1(t)=\lambda t$ for some $\lambda \in \C$ (see, i.e., \cite[Theorem 8.1.11]{manoloal}). Since, $\Re(a_1(t))\geq0$, we have that $\text{Re}(\lambda)\geq0$. By Lemma \ref{hache}, $\lambda=b_{1}$.

Assume now that the functions of the semigroup $\{\Phi_t\}_{t\geq0}$ are not automorphims of $\C_{+}$. Fix $t>0$. There exists $N\geq2$ such that $a_{N}\neq 0$ and 
\[
\varphi_t(s)=a_1+a_NN^{-s}+g(s) 
\]
with $g(s)=o(N^{-\Re(s)})$. 
Given $0<\varepsilon<|a_N|$, we can take $\nu>0$ such that, for every $s\in\C_{\nu}$, it holds $|g(s)|\leq\varepsilon N^{-\Re(s)}$. Assume now that $\Re(a_1(t))=0$. Then,
\[
0<\Re(\varphi_t(s))=\Re(a_1(t))+\Re(a_NN^{-s})+\Re g(s)=\Re(a_NN^{-s})+\Re g(s).
\]
Taking $s_{0}$ such that $\Re(a_NN^{-s_0})=-|a_N|N^{-\Re(s_0)}$ with $\Re s_{0}>\nu$, we have
\begin{equation*}
\begin{split}
0<\Re(a_NN^{-s_0})+\Re g(s_{0})&=-|a_N|N^{-\Re(s_0)}+\Re g(s_{0})\\
&\leq -|a_N|N^{-\Re(s_0)}+\varepsilon N^{-\Re(s_0)}=-N^{\Re(s_0)}(|a_N|-\varepsilon).
\end{split}
\end{equation*}
A contradiction. Therefore, $\Re(a_1(t))>0$ and $\Re b_{1}>0$.
\end{proof}

\begin{remark}
Related to the statement of the above proposition, it is worth mentioning that if $\Phi_{t_{0}}$ is an automorphism of $\C_{0}$ for one $t_{0}>0$, then $\Phi_{t}$ is an automorphism of $\C_{0}$ for all $t\geq 0$ (see \cite[Theorem 8.2.4]{manoloal}).
\end{remark}

Going back to the Koenigs function, we will prove that its derivative is a Dirichlet series. This result lies in the next property of the Banach algebra structure of $\mathcal{H}^{\infty}$. For a proof see, i.e., \cite[Pages 148, 149]{queffelecs}.

\begin{theorem}\label{invertible} 
The invertible elements of $\mathcal{H}^{\infty}$ are the functions $f\in \mathcal{H}^{\infty}$ such that there is $\delta>0$ with  $|f|>\delta$.
\end{theorem}

\begin{theorem}\label{1/H}
Let $\{\Phi_t\}_{t\geq0}$ be a non-trivial continuous semigroup of analytic functions in the class $\mathcal{G}$ and $H$ its infinitesimal generator. Then, $1/H$ is a Dirichlet series.
\end{theorem}
\begin{proof}
We know that $H\in\mathcal{H}^{\infty}(\C_{\sigma})$ for all $\sigma>0$. The statement of Theorem \ref{invertible} holds for this algebra. Let $b_1$ be the first coefficient of $H$. Then, by Proposition \ref{paquito}, $\Re b_1\geq 0$. If the semigroup is formed by automorphisms, then $H$ is constant and the result  is clear. Otherwise, $\Re b_1>0$ and $H:\C_{+}\to \C_{+}$. Hence, $1/H$ is holomorphic.  We know that there is $\nu>0$ such that $|b_{1}-H(s)|<|b_{1}|/2$, whenever $\Re s>\nu$ . So that, $|H(s)|>|b_{1}|/2$. By Theorem \ref{invertible}, $1/H$ is a Dirichlet series in $\C_{\sigma}$. Thus $1/H$ is a Dirichlet series in $\C_{+}$. 
\end{proof}

Bearing in mind the above theorem, equation \eqref{derivadaKoenigs}, and Theorem \ref{gcoro}, we deduce:

\begin{corollary}\label{Koenigs}
Let $h$ be a holomorphic function in $\C_{+}$. Then $h$ is the Koenigs function of a continuous semigroup in $\mathcal G$ if and only if $h'$ is a Dirichlet series satisfying $h'(\C_{+})\subset \overline {\C_{+}}\setminus \{0\}$. \end{corollary}

That is,  $h$ is the Koenigs function of a continuous semigroup in $\mathcal G$ if and only if, up to an additive constant, $h$ belongs to the class $\mathcal{K}$ given by
\begin{equation*}
    \mathcal{K}:=\{h:\C_+\to\C :\,  \Re (h')\geq0, \, h(s)=d_1s-\sum_{n\geq2}\frac{d_n}{\ln ( n)}n^{-s}, \, d_{n}\neq 0 \textrm{ for some } n\}.
\end{equation*}

Many properties of a continuous semigroup can be rewritten in terms of the geometry of the image of the Koenigs function. Denote by $\omega$ the hyperbolic distance in $\C_{+}$. Let us recall that a continuous semigroup with no fixed points in $\C_{+}$ is of {\sl positive hyperbolic step} if there is $s_{0}\in \C_{+}$ such that $\lim_{t\to +\infty} \omega (\Phi_{t}(s_{0}), \Phi_{t+1}(s_{0}))>0$. If such a point exists, then $\lim_{t\to +\infty} \omega (\Phi_{t}(s), \Phi_{t+1}(s))>0$ for all $s$. Otherwise, the semigroup is said to be of {\sl zero hyperbolic step}. This classification provides information about the speed of convergence of the trajectories of the semigroup to the Denjoy-Wolff point. It can be proved that a parabolic semigroup is of zero hyperbolic step if and only if $h(\C_{+})$ is not contained in a horizontal half-plane \cite[Proposition 3.2]{CDP} (see also \cite[Theorem 9.3.5]{manoloal}). 

\begin{lemma}
Given $h\in\mathcal{K}$,
\begin{enumerate}[a)]
 \item If $\Re (d_1)=0$, then $h(s)=d_1s$ and $h(\C_+)$ is a horizontal half-plane.
    \item If $\Re(d_1)>0$, then $h(\C_+)$ contains a non-horizontal half-plane. In particular, it is not contained in a horizontal half-plane.
\end{enumerate}
\end{lemma}
\begin{proof}
We start by proving $a)$. If $\text{Re}(d_1)=0$, then $H=1/h'$ is the infinitesimal generator of a semigroup in the class $\mathcal G$ such that
$$
H(s)=\frac{1}{d_{1}}+\sum_{n\geq 2} c_{n}n^{-s}.
$$
By Proposition \ref{paquito}, $c_{n}=0$ for all $n\geq 2$ and $h(s)=d_{1}s$.  Then $h$ is a dilation followed by a rotation of angle $\pm\pi/2$. Therefore, it takes any vertical half-plane $\C_{\nu}$, $\nu\in\R$, into a horizontal half-plane. 

\ 
Now, we prove $b)$  using Rouch\'e's Theorem. Set $f(s)=d_1s$ and $g(s)=-\sum_{n\geq2}\frac{d_n}{\ln(n)} n^{-s}$. Let us fix $\varepsilon>0$. There exists $\sigma>0$ such that for every $s\in\C_{\sigma}$
\begin{equation*}
    |h(s)-d_1s|=|g(s)|=\left|\sum_{n\geq2}\frac{d_n}{\ln(n)} n^{-s}\right|<\varepsilon.
\end{equation*}
Take $R>0$ such that $|d_{1}|R>\varepsilon$. Fix $\sigma'>\sigma+R$ and $a\in \C_{\sigma'}$.
Set $f_1(s)=h(s)-d_1a$ and $f_2(s)=f(s)-d_1a=d_{1}(s-a)$. Now, for every $s\in\C_{\sigma}$
\begin{equation*}
    |f_1(s)-f_2(s)|=|h(s)-f(s)|=|g(s)|<\varepsilon
\end{equation*}
and, for every $s\in \partial D(a,R)$, 
\begin{equation*}
    |f_2(s)|=|f(s)-d_1a|=|d_1(s-a)|=|d_1|R.
\end{equation*}
Thus
\[
|f_1(s)-f_2(s)|<|f_2(s)|,\quad \text{for every $s\in\partial D(a,R)$}.
\]
Then, by Rouch\'e's Theorem, $f_1$ and $f_2$ have the same number of zeros inside $D(a,R)$. 
Hence, since $f_2(a)=0$, we deduce that there exists $c\in D(a,R)\subset \C_{\sigma}$, such that $f_1(c)=0$. Then, $h(c)=d_1a$. Therefore  $h(\C_{+})$ contains $d_{1}\C_{\sigma'}$. Since $\Re d_{1}>0$, $d_{1}\C_{\sigma'}$ is a non-horizontal half-plane.
\end{proof}

\begin{proposition}\label{Prop:zerohypstep}
Let $\{\Phi_{t}\}$ be a continuous semigroup in the class $\mathcal G$ which are not automorphims of $\C_{+}$. Then the semigroup is of zero hyperbolic step. 
\end{proposition}

We end this section  with an example showing that in general the image of the Koenigs map of a continuous semigroup in $\mathcal G$ is not contained in a half-plane. 
\begin{example} \label{exampleluis} By Theorem \ref{gcoro}, the function $H(s)=2^{-s}+1$ is the infinitesimal generator of a continuous semigroup in $\mathcal G$. By \eqref{derivadaKoenigs}, its Koenigs map is given by 
\[
h(s)=\frac{\text{log} (1+2^{s})}{\ln 2}=s+\frac{\text{Log}(1+2^{-s})}{\ln 2},
\]
where  $\text{log}$ means a suitable continuous branch of the logarithm and $\text{Log}$ denotes the main branch of the logarithm. 
For all $k\in \Z$, the horizontal line 
\[
\{
s\in\C: s=x+iy, \  x\in\R , \  y=\frac{(2k+1)\pi}{\ln 2}
\}
\]
is contained in $h(\C_+)$, so that $h(\C_{+})$ cannot be contained in a half-plane. The reader can check that 
$$
\Phi_{t}(s)=\frac{1}{\ln 2}\log(2^{t}-1+2^{s+t})=s+t+\frac{1}{\ln 2}\text{Log}(1+2^{-s}(1-2^{-t})), \quad t\geq 0, \ s\in \C_{+}.
$$
\end{example}

\section{Semigroups of composition operators in \text{$\mathcal{H}^{\infty}$}}\label{sec:hinfinito}

Theorem \ref{boundedness-Hinfty} and Proposition \ref{Prop:algebraicsemigroup} show that given an algebraic semigroup of composition operators $\{T_t\}_{t\geq0}=\{C_{\Phi_{t}}\}_{t\geq0}$  in $\mathcal{H}^{\infty}$, the family of functions $\{ \Phi_{t}\}_{t\geq0}$ is an algebraic semigroup of analytic functions in $ \C_{+}$ satisfying that 
$$
 \Phi_{t}(s)=c_{\Phi_{t}}s+\varphi_{t}(s), \quad s\in \C_{+}
 $$
 where  $c_{\Phi_{t}}\in\N\cup\{0\}$ and $\varphi_{t}\in\mathcal{D}$. 
 In this section we prove that 
\begin{theorem}\label{gordoinfty}
Let $\{T_t\}_{t\geq0}$ be a strongly continuous semigroup of composition operators in $\mathcal{H}^{\infty}$. Then, $T_t=\mathrm{Id}$ for every $t\geq0$.
\end{theorem}

We need a preliminary lemma whose proof is a slight adaptation of ``{\sl a)} implies {\sl b)}'' in Theorem \ref{gordo1} so that we omit it:
\begin{lemma} \label{lem:gordoinfty}
Let $\{\Phi_{t}\}$ be a semigroup in $\mathcal G_{\infty}$.
If  $\{T_t\}_{t\geq0}=\{C_{\Phi_{t}}\}_{t\geq0}$   is a strongly continuous semigroup in $\mathcal{H}^\infty$, 
then $\{\Phi_t\}_{t\geq0}$ is a continuous semigroup in $\mathcal G$.
\end{lemma}

\begin{proof}[Proof of Theorem \ref{gordoinfty}]
Write $\{T_t\}_{t\geq0}=\{C_{\Phi_{t}}\}_{t\geq0}$.
By Lemma \ref{lem:gordoinfty}, $\{\Phi_t\}_{t\geq0}$ is a continuous semigroup in $\C_{+}$. Therefore, it has an infinitesimal generator $H:\C_{+}\to \overline{\C_{+}}$. $H$ is an holomorphic function in $\C_{+}$.  Assume that $H$ is non-zero. Take $L$ a linear fractional map sending the unit disc onto the right half-plane. Then $H\circ L$ is a holomorphic function from $\D$ into $\overline{\C_{+}}$. Then it has non-tangential limit at almost every point in the boundary of the  unit disc and, since it is non-zero, there is a point in the boundary of the unit disc such that the limit is a complex number different from zero (see, i.e., \cite[Theorems 3.2 and 2.2]{duren_theory_2000}).
 Then, there are $y\in\R$, $\varepsilon>0$, and $\delta>0$ such that
\begin{equation}\label{priesa}
|H(x+iy)|>\delta,
 \end{equation}
for all $x\in (0,\varepsilon$). 

A standard argument shows that the infinitesimal generator of the semigroup $\{T_{t}\}$ is given by $A(f)=Hf'$ and, arguing as we did in the proof of Proposition \ref{Prop:infgen}, its domain is
\begin{equation*}
    D(A):=\{ f\in\mathcal{H}^{\infty}:Hf'\in\mathcal{H}^{\infty}
    \}.
\end{equation*}
We claim that for every $f\in D(A)$ there exists 
\begin{equation*}
    \lim_{x\to0^+}f(x+iy).
\end{equation*}
Let us proof this claim. 
Take $f\in D(A)$. Therefore, there exists $M>0$ such that
\begin{equation*}
    |H(s)f'(s)|<M, \quad s\in\C_+.
\end{equation*}
Now,  given $s=x+iy$ such that $x\in(0,\varepsilon)$, putting together this bound and \eqref{priesa},
\begin{equation*}
    \delta|f'(s)|<M.
\end{equation*}
Then, 
\begin{equation}\label{dct}
    \sup_{x\in(0,\varepsilon)}|f'(x+iy)|\leq \frac{M}{\delta}.
\end{equation}
This estimate will allow us to apply the Dominated Convergence Theorem to \begin{align}\label{aliño}
    f(x+iy)=f(\varepsilon+iy)-\int_x^{\varepsilon}f'(u+iy)du.
\end{align}
Indeed, \eqref{dct} and the fact that $\chi_{(x,\varepsilon)}\to\chi_{(0,\varepsilon)}$ as $x\to0^+$ allows us to apply the Dominated Convergence Theorem to the integral in \eqref{aliño}. Taking the limit when $x\to0^+$, we find that
\[
\lim_{x\to0^+}f(x+iy)=\lim_{x\to0^+}\left(f(\varepsilon+iy)-\int_x^{\varepsilon}f'(u+iy)du\right)=f(\varepsilon+iy)-\int_0^{\varepsilon}f'(u+iy)du.
\]
Therefore, the above claim holds.

Take now a function $f\in \overline{D(A)}^{\|\cdot\|_{\mathcal{H}^{\infty}}}$. Again we have that \begin{equation*}
    \lim_{x\to0^+}f(x+iy)
\end{equation*}
exists. 
Let us assume on the contrary that such  limit does not exist. Then, there would exist two sequence of positive  real numbers $\{u_n\}_{n\in\N}$ and $\{v_n\}_{n\in\N}$ both tending to zero and $\lambda_1,\lambda_2\in\C$, $\lambda_1\not=\lambda_2$ such that
\begin{equation*}
    \lim_{n\to\infty}f(u_n+iy)=\lambda_1\quad \text{and}\quad \lim_{n\to\infty}f(v_n+iy)=\lambda_2.
\end{equation*}
Set $\kappa=\frac13|\lambda_1-\lambda_2|>0$. Since $f\in \overline{D(A)}^{\|\cdot\|_{\mathcal{H}^{\infty}}}$, there exists $g\in D(A)$ such that
\[
\|f-g\|_{\mathcal{H}^{\infty}}<\kappa.
\]
Let $\lambda:=\lim_{x\to0^+}g(x+iy)$ (the limit exits by above claim). Then, for every $n\in\N$
\begin{equation*}
    |f(u_n+iy)-g(u_n+iy)|<\kappa.
\end{equation*}
Letting $n\to\infty$, we have that $|\lambda_1-\lambda|\leq\kappa$. An identical argument for $\{v_n\}_n$ instead of $\{u_n\}_{n}$ allows to deduce that $|\lambda_2-\lambda|\leq\kappa$. However, this contradicts the choice we did of $\kappa$. Therefore, for each $f\in \overline{D(A)}^{\|\cdot\|_{\mathcal{H}^{\infty}}}$, there exists \begin{equation*}
    \lim_{x\to0^+}f(x+iy).
\end{equation*}

In virtue of \cite[Theorem 1.4, page 37]{engels}, in order to get a contradiction with the fact that $H$ is non-zero,  it suffices to find a function in $\mathcal{H}^{\infty}$ which does not lie in the closure of the set $D(A)$.  Consider the holomorphic function in the unit disc
\[
F(z)=\exp\left(i\log\left( \frac{1+z}{1-z}\right)\right).
\]
Since $F$ is bounded, we have that  $f(s)=F(2^{-s+iy})$, $s\in \C_{+}$, belongs to $\mathcal{H}^{\infty}$. But the limit 
$   \lim_{x\to0^+}f(x+iy)$ does not exist so that $f\notin  \overline{D(A)}^{\|\cdot\|_{\mathcal{H}^{\infty}}}$. Therefore, $H$ is the null function and the semigroup is the trivial one.
\end{proof}

\end{document}